\documentclass[aos, preprint]{imsart}

\usepackage{url,color}
\usepackage{graphicx}
\usepackage{amsmath, amsthm, amssymb}
\usepackage{algpseudocode}
\usepackage{algorithm}
\usepackage{subfigure}
\usepackage{comment}

\usepackage{multirow}
\let\hat\widehat

\newtheorem{theorem}{Theorem}
\newtheorem{lemma}[theorem]{Lemma}

\newtheorem{proposition}[theorem]{Proposition}

\newtheorem{theorem1}{Theorem}
\newtheorem{remark}[theorem1]{Remark}

\newtheorem{theorem2}{Theorem}
\newtheorem{example}[theorem2]{Example}

\usepackage{xr}
\newcommand\R{\mathbb{R}}

\newcommand\E{\mathbb{E}}

\newcommand\norm[1]{\|#1\|}
\newcommand\Haus{{\sf Haus}}

\DeclareMathOperator{\reach}{{\sf reach}}

\startlocaldefs
\endlocaldefs

\begin{document}

\begin{frontmatter}

\title{Solution manifold
and Its Statistical Applications
}
\runtitle{Solution manifold}


\author{\fnms{Yen-Chi} \snm{Chen}\ead[label=e1]{yenchic@uw.edu}}
\address{Department of Statistics\\University of Washington\\ \printead{e1}}
 \thankstext{e1}{Supported by NSF grant DMS - 195278 and DMS - 2112907 and NIH grant U24 AG072122} 

\runauthor{Chen}

\begin{abstract}
A solution manifold is the collection of points in a $d$-dimensional space satisfying a system of $s$ equations with $s<d$. 
Solution manifolds occur in several statistical problems including
missing data, 
algorithmic fairness,
hypothesis testing, 
partial identifications, and nonparametric set estimation. 
We  theoretically and algorithmically analyze solution manifolds. 
In terms of theory, we derive five useful results:
smoothness theorem,  stability theorem (which implies the consistency of a plug-in estimator),
convergence of a gradient flow,
local center manifold theorem
and convergence of the gradient descent algorithm.
We propose a Monte Carlo gradient descent algorithm to numerically approximate a solution manifold.
In the case of the likelihood inference, we design a manifold constraint maximization procedure to
find the maximum likelihood estimator on the manifold.
\end{abstract}





\end{frontmatter}


\section{Introduction}

A solution manifold \cite{Rheinboldt1988}
is the collection of points in $d$-dimensional space
that solves a system of $s$ equations
where $s<d$. 
Namely, feasible set is a collection of points in an under-constrained system. 
Under smoothness conditions, the feasible set forms a manifold known as a solution manifold.

The solution manifold
occurs in many problems in statistics
such as missing data (Example 1), algorithmic fairness (Example 2), 
constrained likelihood space (Example 3), and density ridges/level sets (Example 4).
In the regular case that $s=d$, the solution manifold
reduces to the usual problems such as the Z-estimators \cite{vdv1998} or estimating equations \cite{liang1986longitudinal}. 
While there has been a tremendous amount of literature on 
the analysis of regular cases ($s=d$),
little is known 
when $s<d$. 
This study aims to analyze the problem when $s<d$
and design a practical algorithm to find the manifold.



Formally, 
let $\Psi:\R^d \mapsto \R^s$ be a vector-valued function with $s<d$. 
The solution set of $\Psi$ 
$$
M = \{x:\Psi(x) = 0\}
$$
is called the solution manifold and we call $\Psi$ the generator (function) of $M$.
Note that in some applications, $x$ represents the parameter in a model;
thus, sometimes we write $M = \{\theta: \Psi(\theta) = 0\}$.
Here we provide examples of solution manifolds from
various statistical problems.



\begin{figure}
\center
\includegraphics[height=1.2in]{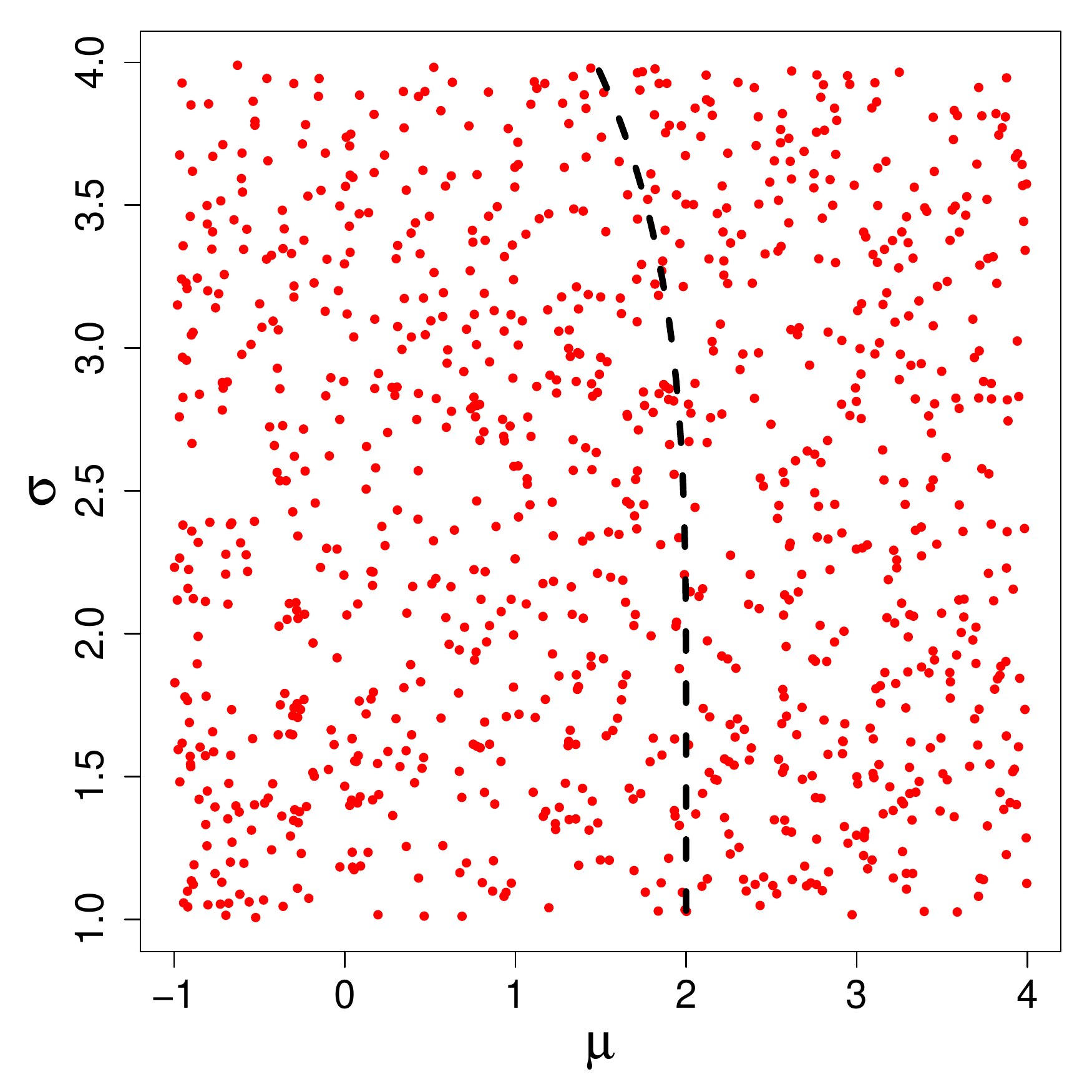}
\includegraphics[height=1.2in]{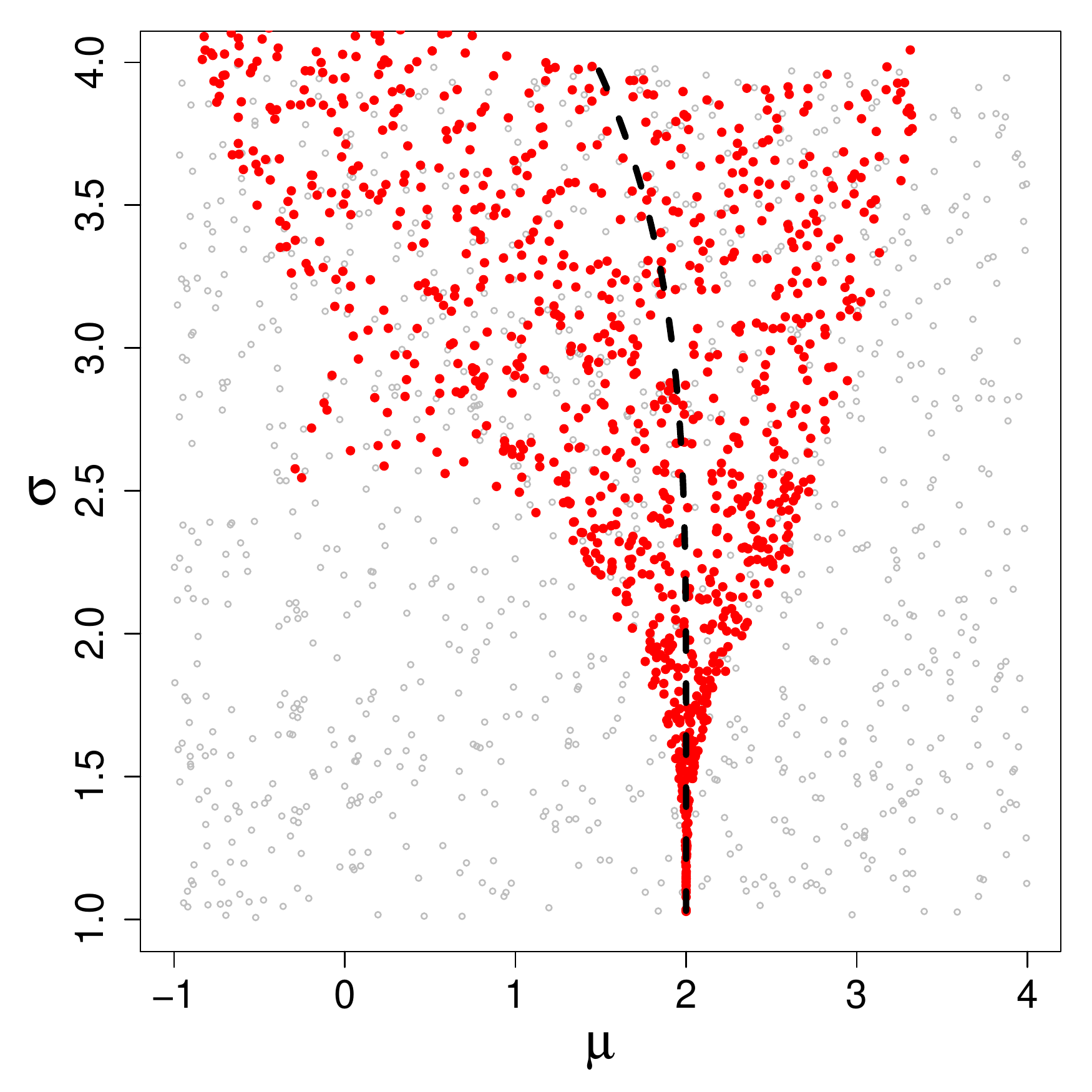}
\includegraphics[height=1.2in]{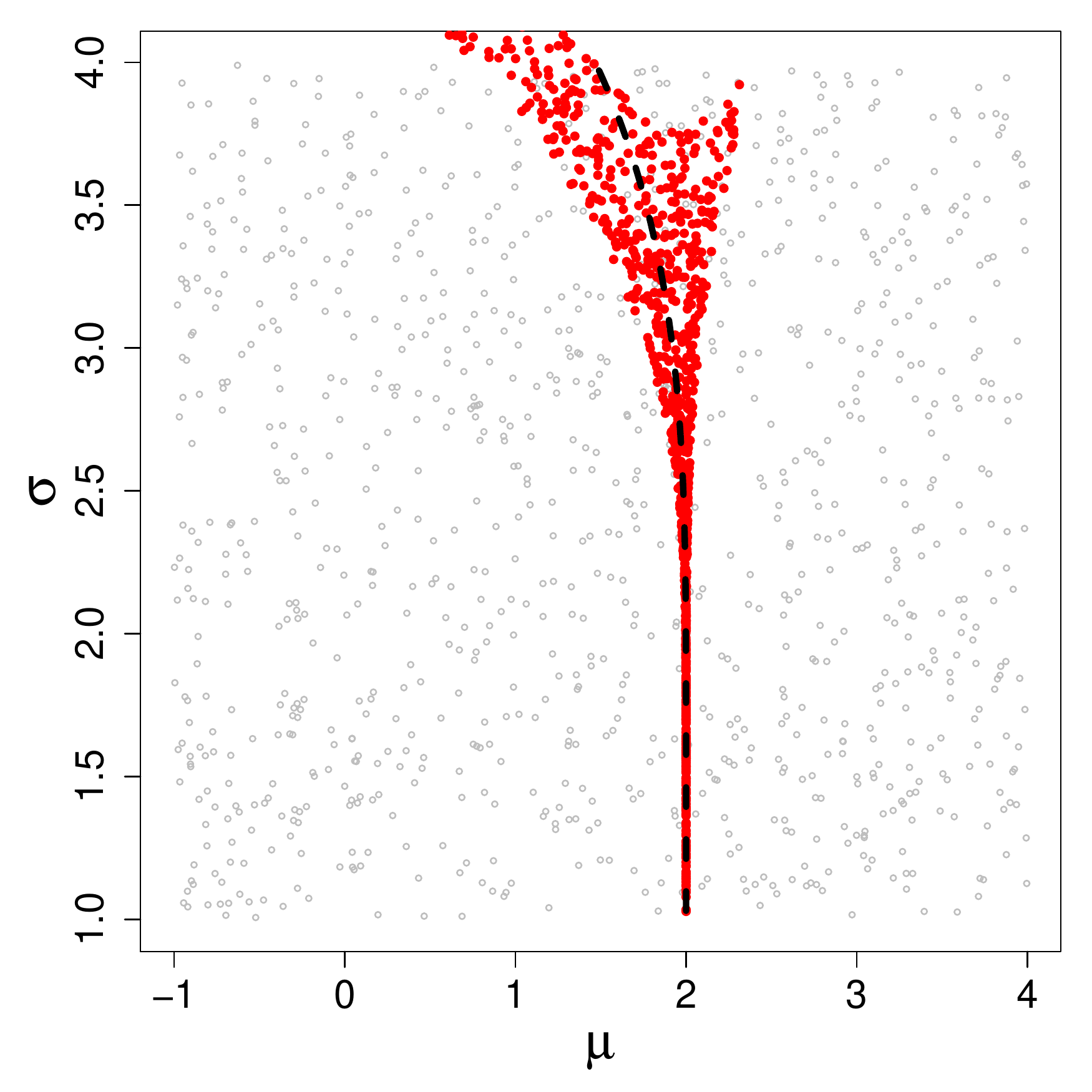}
\includegraphics[height=1.2in]{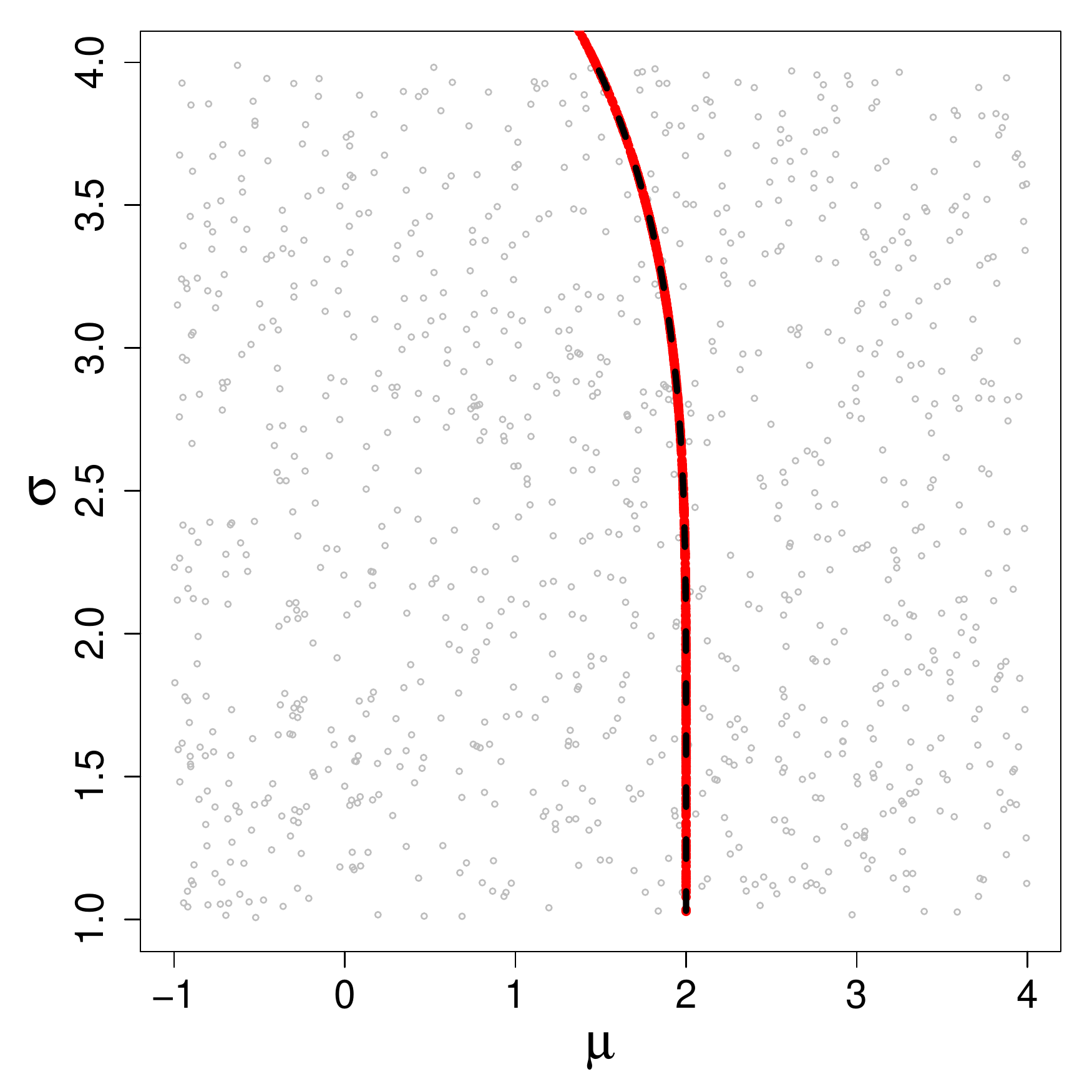}
\caption{An example of a solution manifold formed by the parameter $(\mu,\sigma)$
of a Gaussian
with a tail probability bound $P(-5<Y<2) = 0.5$.
The left panel shows $1000$ random initializations (uniformly distributed within $[1,3]\times [2,4]$). We keep applying the
gradient descent algorithm until convergence (right panel). The black dashed line indicates the actual location
of the solution manifold.
}
\label{fig::ex02}
\end{figure}

\begin{example}[Missing data]	\label{ex::missing}
Consider a simple missing data problem where we have a binary response variable $Y$
and a binary covariate $X$. The response variable is subject to missing. We use a binary variable $R$
to indicate the response pattern of $Y$ (i.e., $Y$ is observed if $R=1$).
Depending on the value of $R$, we may observe $(X,Y,R=1)$ or $(X,R=0)$. 
In this case, the entire distribution is characterized by the following parameters:
$$
\zeta_{x,y} = P(R=1|X=x,Y=y),\quad \mu_x = P(Y=1|X=x) ,\quad \xi =  P(X=x)
$$
for $x,y\in\{0,1\}$.
Parameter $\zeta_{x,y}$ is called missing data mechanism \cite{little2019statistical}.
$\mu_x$ is the regression function, and $\xi$ is the marginal mean of $X$.
Thus, this problem has seven parameters. 
From the observed data (IID elements in the form of $(X,Y,R=1)$ or $(X,R=0)$),
we can identify $P(x,y,R=1)$ and $P(x,R=0)$ for $x,y\in\{0,1\}$,
which leads to six constraints (note that $P(x,y,r) = P(X=x,Y=y,R=r)$):
\begin{equation}
\begin{aligned}
P(1,1,1) &= \zeta_{11}\mu_1\xi,\qquad\qquad P(1,0,1) = \zeta_{10}(1-\mu_1)\xi,\\
P(0,1,1) &= \zeta_{01}\mu_0(1-\xi),\qquad P(0,0,1) = \zeta_{00}(1-\mu_0)(1-\xi)\\
P(X=0,R=0)& = (1-\zeta_{01})\mu_0(1-\xi) + (1-\zeta_{00})(1-\mu_0) (1-\xi)\\
P(X=1,R=0)& = (1-\zeta_{11})\mu_1\xi + (1-\zeta_{10})(1-\mu_1) \xi.
\end{aligned}
\label{eq::missing::psi}
\end{equation}
Thus, the feasible values of  the parameters $(\zeta_{x,y}, \mu_x,\xi)$ will 
form a solution manifold
and the above constraints describe the generator $\Psi$.
Note that the resulting solution manifold is related to the nonparametric bound of the parameters \cite{manski1990nonparametric, manski1999identification, chen2018monte};
see Section~\ref{sec::missing}
for more discussions.

\end{example}

\begin{example}[Algorithmic fairness]	\label{ex::fair}
The algorithmic fairness is a trending topic in modern machine learning research \cite{hardt2016equality, 
chouldechova2017fair, corbett2017algorithmic}.
We consider a post-processing method in
the algorithmic fairness study. 
Suppose we have a binary response $Y\in\{0,1\}$,
a sensitive binary variable $A\in\{0,1\}$ that we wish to protect,
and an output from a trained classifier $W\in\{0,1\}$ (one can view it as $W= c(X,A)$, where $X$ is the covariate/feature
and $c$ is a trained classifier).
The sensitive variable is often the race or gender indicator. 
The classification result based on $W$ may discriminate against the sensitive variable $A$; that is, 
it is likely that $A=1$ and $W=1$ occur at the same time.
We want to construct a new `fair' classifier $Q\in\{0,1\}$  such that $Q$ is a random variable whose distribution depends only
on $A$ and $W$ and the output of $Q$
will not discriminate against $A$.
In other words, we design a new random variable $Q$ such that $Q\perp Y|A,W$.
While there are many principles of algorithmic fairness, we consider the \emph{test fairness} \cite{chouldechova2017fair}: We want to construct $Q$
such that 
\begin{equation}
P(Y=1|Q=s,A=0) = P(Y=1|Q=s, A=1),
\label{eq::fair}
\end{equation}
for each $s=0,1$.
To construct $Q$ that satisfies the above constraint,
we generate $Q$ based on $A,W$ such that its distribution is determined by parameter $q_{w,a} = P(Q=1|W=w,A=a)$.
As long as we can properly choose $q_{w,a}$,
the resulting $Q$ will satisfy equation \eqref{eq::fair}. 
In this case, 
any $q_{w,a}$ solving the following two equations will satisfy equation \eqref{eq::fair} (see Appendix \ref{sec::fair}):
\begin{equation}
\begin{aligned}
\frac{\sum_w q_{w,0}P(W=w,Y=1|A=0)}{\sum_{w'}q_{w',0}P(W=w'|A=0)} &= \frac{\sum_w q_{w,1}P(W=w,Y=1|A=1)}{\sum_{w'}q_{w',1}P(W=w'|A=1)}.\\
\frac{\sum_w (1-q_{w,0})P(W=w,Y=1|A=0)}{\sum_{w'}(1-q_{w',0})P(W=w'|A=0)} &= \frac{\sum_w (1-q_{w,1})P(W=w,Y=1|A=1)}{\sum_{w'}(1-q_{w',1})P(W=w'|A=1)}.
\end{aligned}
\label{eq::fair2}
\end{equation}
Note that $P(W=w,Y=y,A=a)$ is identifiable from the data.
The original parameter $\{q_{w,a}:w,a=0,1\}$ is in four-dimensional space
and we have two constraints, leading to a solution manifold of two dimensions.

\end{example}

\begin{example}[Constrained likelihood space]
Consider a random variable $Y$ from an unknown distribution.
We place a 
parametric model $p(y;\theta)$ of the underlying PDF of $Y$ where $\theta\in\R^d$
is the parameter vector. 
Suppose that we have a set of constraints on the model
such that a feasible parameter must satisfy
$$
f_1(\theta)= f_2(\theta) = \cdots = f_s(\theta) = 0
$$
for some given functions $f_1,\cdots, f_s$.
These functions may be from  independence assumptions or moment constraints $\E(g_j(Y)) = 0$ for some functions $g_1,\cdots, g_s$.
The set of parameters that satisfies these constraints is 
\begin{equation}
\Theta_0 = \left\{\theta: f_\ell(\theta)= 0, \ell=1,\cdots, s\right\} = \left\{\theta: \Psi(\theta) = 0 \right\},
\label{eq::moment1}
\end{equation}
which is a solution manifold with $\Psi_\ell(\theta) = \int f_\ell(y) p(y;\theta)dy$.
The above model is used in algebraic statistics \cite{drton2007algebraic,michalek2016exponential}, 
partial identification problems with equality constraints \cite{hansen1982large,chernozhukov2007estimation},
and
mixture models with moment constraints \cite{lindsay1995mixture, chauveau2013ecm}.
We will return to this problem in 
Section
\ref{sec::PI}.
Figure~\ref{fig::ex02} shows an example of a solution manifold formed by the tail probability constraint $P(-5<Y<2) = 0.5$
where $Y\sim N(\mu,\sigma^2)$. 
We have  two parameters $(\mu, \sigma^2)$ and one constraint; thus, the resulting solution set is a one-dimensional manifold.
From the left to the right panels, we show that we can recover the underlying manifold by random initializations 
with a suitable gradient descent process (Algorithm~\ref{alg::alg}). 
\end{example}
\begin{example}[Density ridges]
A $k$-ridge \cite{genovese2014nonparametric} of a density function $p(x)$ 
is defined as the collection of points satisfying
$$
\{x: V_k(x)^T\nabla p(x) = 0, \lambda_k(x)<0\},
$$
where $V_k(x) = [v_k(x),\cdots, v_{d}(x)]\in\R^{(k-d)}$ is the collection of eigenvectors of $H(x) = \nabla \nabla p(x)$,
and $\lambda_k(x)$ is the $k$-th eigenvector. The eigen-pairs are ordered as $\lambda_1(x)\geq \lambda_2(x)\geq \cdots \geq \lambda_d(x)$.
In this case $\Psi(x) = V_k(x)^T\nabla p(x)$;
hence, ridges are also solution manifolds.
In addition to ridges,  the level sets 
and critical points of a function \cite{Walther1997,mammen2013confidence,chacon2015population}
are examples of solution manifolds. 
We will discuss this in Section~\ref{sec::NP}
\end{example}

Although the aforementioned examples are from different statistical problems,
they all share a similar structure that
the feasible set forms a solution manifold.
Thus, we study properties of solution manifolds in this paper,
and our results will be applicable to all of these cases.

\emph{Main results.}
Our main results include theoretical developments and algorithmic innovations. 
In the theoretical analysis, we show that under a similar sets of assumptions,  we have the following:
\begin{enumerate}
\item {\bf Smoothness theorem.} 
The solution manifold
is a $(d-s)$-dimensional manifold with a positive reach (Lemma~\ref{lem::normal} and Theorem~\ref{thm::SM2}). 
\item {\bf Stability theorem.}  As long as $\hat \Psi$ and $\Psi$ and their derivatives are sufficiently close, $\hat M = \{x: \hat \Psi(x) = 0\}$
converges to $M$ under the Hausdorff distance (Theorem~\ref{thm::LU0}). 
\item {\bf Convergence of a gradient flow.} 
For the gradient descent flow of $\|\Psi(x)\|^2$, the flow converges (in the normal direction of $M$) to a point on $M$ when the starting point is sufficiently close to $M$
(Theorem~\ref{thm::GD}).

\item {\bf Local center manifold theorem.} 
The collection of points converging to the same location $z\in M$ forms an $s$-dimensional manifold (Theorem~\ref{thm::AMT}). 
\item {\bf Convergence of a gradient descent algorithm.} 
With a good initialization,
the gradient descent algorithm of $\|\Psi(x)\|^2$ converges linearly to a point in $M$ (Theorem~\ref{thm::alg})
when the step size is sufficiently small.

\end{enumerate}
We propose
three algorithms
to numerically find solution manifolds and use them to handle statistical problems:
%
\begin{enumerate}
\item {\bf Monte Carlo gradient descent algorithm:} an algorithm generating points over $M$
that requires only the access to $\Psi$ and its gradient (Section~\ref{sec::MCGD} and Algorithm~\ref{alg::alg}). 
\item {\bf Manifold-constraint maximizing algorithm: }an algorithm that finds the MLE on the solution manifold (Section~\ref{sec::LRT} and Algorithm~\ref{alg::alg::MLE}). 
\item {\bf Approximated manifold posterior algorithm: }a Bayesian procedure that approximates the posterior distribution on a manifold (Appendix~\ref{sec::bayesian} and Algorithm~\ref{alg::alg::AMP}).
\end{enumerate}

We would like to emphasize that while some of the theoretical results have appeared in the regular case ($s=d$, i.e.,
the solution manifold is a collection of points or just a single point), generalizing these
results to the manifold cases ($s<d$) requires non-trivial extensions of existing techniques.
The major challenge comes from the fact that the set $M$ contains an infinite number of points
and the geometry of $M$ will pose technical issues during the theoretical analysis. 
Also, although the stability theorem has appears for specific examples such as level set and ridge estimation 
\cite{Cadre2006,rinaldo2012stability, genovese2014nonparametric}, 
there is no such result for the general class of solution manifolds.
This paper provides a unified framework of analyzing solution manifolds.

The impact of this paper is beyond statistics. 
Our result provides a new analysis of the partial identification problem in econometrics \cite{hansen1982large,chernozhukov2007estimation}. 
The local center manifold theorem offers a new class of statistical problems
where the dynamical system interacts with statistics \cite{perko2013differential}.
The Monte Carlo approximation of a solution manifold leads to 
a point cloud over the manifold, which is a common scenario in computational geometry \cite{Cheng2005,cheng2005manifold,dey2006curve}. 
The algorithmic convergence of the gradient descent demonstrates a new class of non-convex functions
for that we still obtain the linear convergence \cite{boyd2004convex,nesterov2018lectures}. 

\emph{Outline.}
The remainder of this paper is organized as follows.
Section~\ref{sec::solM} provides a formal definition of a solution manifold and
studies the smoothness and stability of the manifold. 
Section~\ref{sec::MCGD} presents an algorithm for
approximating the solution manifold
and an analysis of its properties. 
Section~\ref{sec::app} discusses several statistical applications of solution manifolds.
Section~\ref{sec::discuss} provides future directions,
connections with other fields, and some manifolds in statistics that are not in a solution form.

\emph{Notations.}
Let $v\in\mathbb{R}^d$ be a vector and $V\in\mathbb{R}^{n\times m}$ be a matrix. 
$\norm{v}_2$ is the $L_2$ norm (Euclidean norm) of $v$, and 
$\norm{v}_{\max} = \max\{|v_1|,\cdots,|v_d|\}$ is the vector max norm.
For matrices, we use
$\norm{V} = \norm{V}_2 = \max_{\norm{u}=1, u\in \mathbb{R}^m}\frac{\norm{Vu}_2}{\norm{u}_2}$
as the $L_2$ norm 
and 
$\norm{V}_{\max} = \max_{i,j} \norm{V_{ij}}$
as the max norm.
For a squared matrix $A$, we define $\lambda_{\min}(A), \lambda_{\max}(A)$ to be the minimal and maximal eigenvalue respectively, and
$\lambda^2_{\min, >0}(A)$ as the smallest non-zero eigenvalue.
For a vector value function $\Psi$, we define a maximal norm of derivatives as 
$$
  \norm{\Psi}_{\infty}^{(J)} = \sup_x\max_i \max_{j_1}\cdots \max_{j_J} \left|\frac{\partial^J}{\partial x_{j_1}\cdots\partial x_{j_J}}\Psi_i(x)\right|
$$
for $J=0,1,2,3.$
When $\Psi$ is a scalar function, 
this is reduced to 
$$
  \norm{\Psi}_{\infty}^{(1)} = \sup_x \|\nabla\Psi(x)\|_{\max},\quad   \norm{\Psi}_{\infty}^{(2)} = \sup_x \|\nabla\nabla \Psi(x)\|_{\max},
$$
which
are the usual maximal norm of the gradient vector and Hessian matrix over all $x$.
We also define
\begin{align*}
\norm{\Psi}^*_{\infty, J} = \underset{j=0,\cdots,J}{\max} \norm{\Psi}^{(j)}_{\infty}
\end{align*}
as a norm that measures distance using upto the $J$-th derivative.
The Jacobian (gradient) of $\Psi(x)$ is an $s\times d$ matrix
$$
G_\Psi(x) = \nabla \Psi(x) = \begin{bmatrix}
\nabla \Psi_1(x)^T\\
\nabla \Psi_2(x)^T\\
\hdots\\
\nabla \Psi_s(x)^T
\end{bmatrix}\in \R^{s\times d}
$$
and the Hessian of $\Psi(x)$ will be an $s\times d\times d$ array 
$$
H_\Psi(x) = \nabla \nabla\Psi(x)\in \R^{s\times d\times d}, \quad [H_\Psi(x)]_{ijk} = \frac{\partial^2}{\partial x_j\partial x_k}\Psi_i(x)
$$
and third derivative of $\Psi$ will be an array 
$$
\nabla \nabla\nabla \Psi(x)\in \R^{s\times d\times d\times d}, \quad [\nabla \nabla\nabla \Psi(x)]_{ijk\ell} = \frac{\partial^3}{\partial x_j\partial x_k\partial x_\ell}\Psi_i(x).
$$
Let $A$ be a set and $x$ be a point.
We then define 
$$
d(x,A) = \inf\{\|x-y\|: y\in A\}
$$
as the projected distance from $x$ to $A$.
For a set $A$ and a positive number $r$, we denote $A\oplus r = \{x: d(x,A)\leq r\}$.

\section{Solution manifold and its geometry}	\label{sec::solM}

Let $\Psi:\R^d \mapsto \R^s$
be a vector-valued function and $M  = \{x: \Psi(x) = 0\}$ be the solution set/manifold.
When the Jacobian matrix $G_\Psi(x) = \nabla\Psi(x) $
has rank $s$ at every $x\in M$,
the set $M$ is an $(d-s)$-dimensional manifold locally at every point $x$
due to the implicit function theorem \cite{rudin1964principles}.
In algebraic statistics, the parameters in the solution set $\{x: \Psi(x)= 0 \}$
are called an implicit (statistical) algebraic model \cite{gibilisco2010algebraic}.

For a solution manifold,
its normal space can be characterized using the following lemma.
\begin{lemma}
For every point $x\in M$, the row space of $G_\Psi(x) \in \R^{s\times d}$ 
spans the normal space of $M$ at $x$. 
\label{lem::normal}
\end{lemma}

Lemma~\ref{lem::normal} is an elementary result from geometry (see Section 6.5.1 of \cite{garrity2001all});
hence, we omit its proof.
Lemma~\ref{lem::normal} states that the Jacobian/gradient of  $\Psi$
is normal to the solution manifold. 
This is a natural result because the gradient of a function is always normal
to the level set and the solution manifold
can be viewed as the intersection of level sets of different functions. 

\subsection{Assumptions}	\label{sec::assumption}
We will make the following two major assumptions in this paper.
All the theoretical results rely on
these two assumptions.
\begin{description}
\item \textbf{(D-k)} $\Psi(x)$ is bounded $k$-times differentiable.
\item \textbf{(F)} There exists $\lambda_0,\delta_0,c_0>0, $ such that 
	\begin{itemize}
	\item[A.] $\lambda_{\min}(G_{\Psi}(x)G_{\Psi}(x)^T)\equiv \lambda_{\min,>0}(G_{\Psi}(x)^TG_{\Psi}(x)) >\lambda_0^2$ for all $x\in M\oplus \delta_0$ and
	\item[B.] $\norm{\Psi(x)}_{\max}>c_0$ for all $x\notin M\oplus \delta_0$.
	\end{itemize}
\end{description}

Assumption (D-k) is an ordinary smoothness of the generator function. 
It may be relaxed by a H{\" o}lder type condition on $\Psi(x)$.
Assumption (D-k) is weaker for a smaller integer $k$.
In the stability analysis, we need (D-1)  (i.e.,  bounded first order derivative) condition.
If we want to control the smoothness of the manifold, we need (D-2) or a higher-order condition. 
The stability of a gradient flow requires a (D-3) condition.
Assumption (F)
is a curvature assumption of $\Psi$ around the solution manifold.
Lemma~\ref{lem::normal} implies that the normal space of $M$
at each point is well-defined.
This assumption will reduce to  commonly assumed conditions
in the literature.
For instance, 
in the case of mode estimation (finding the local modes of a PDF $p(x)$), 
(F) reduces to the assumption that the local modes are well-defined as separated \cite{romano1988bootstrapping,romano1988weak}.
This is 
similar to the assumption that the PDF $p(x)$ is a Morse function \cite{chen2016comprehensive,jisu2016statistical}.
In the MLE theory, (F) refers to the Fisher's information matrix being positive definite at the MLE (and other local maxima),
which is often viewed as a classical condition in the MLE theory (see, e.g., Chapter 5 of \cite{vdv1998}).
In the problem of finding the density level sets (finding the set $\{x: p(x) =\lambda\}$),
this assumption is equivalent to assuming that $p(x)$ has a non-zero gradient  around the level set
\cite{molchanov1991empirical,Tsybakov1997, Molchanov1998, Cadre2006, mammen2013confidence,Laloe2012}.
The assumption (F) is critical to the that the set $M$ forms a manifold.
When there is no lower bound $\lambda_0$, i.e., the gradient $\lambda_{\min}(G_{\Psi}(x)G_{\Psi}(x)^T)$
attains $0$, the set $M$ may not form a manifold. 
One scenario that this occurs is the density level set $\{x: p(x) = \lambda\}$ such that
the density function $p(x)$ is flat at the level $\lambda$.

The constants in (F) can be further characterized by the following lemma.
\begin{lemma}
Assume (D-2)  and
that there exists $\lambda_M >0$ such that 
$$
(F')\qquad\inf_{x\in M}\lambda_{\min}(G_{\Psi}(x)G_{\Psi}(x)^T)\geq\lambda_M^2.
$$
Then the constants in
(F) can be chosen as
$$
\lambda_0 = \frac{1}{2}\lambda_{M},\quad 
\delta_0 =\frac{3\lambda^2_M}{8\|\Psi^*_{\infty, 1}\|\|\Psi^*_{\infty, 2}\|},\quad
c_0= \inf_{x\notin M\oplus \delta_0}\|\Psi\|_{\max}.
$$
\label{lem::relax}
\end{lemma}

Lemma~\ref{lem::relax} only places assumptions on 
the behavior of $\Psi$ and its derivatives on
the manifold $M$. 
(F') is the eigengap conditions for the row space of $G_{\Psi}(x)$. 
The assumption in Lemma~\ref{lem::relax} is very mild. 
When estimating the local modes of a function,
the assumption is the same as requiring that the Hessian matrix at local modes 
have all eigenvalues being positive.
The requirement $\inf_{x\in M}\lambda_{\min}(G_{\Psi}(x)G_{\Psi}(x)^T)\geq\lambda_M^2$
implies that
rows of $G_{\Psi}(x)$ are linearly independent for all $x$.


\subsection{Smoothness of Solution Manifolds}	\label{sec::SSM}

We introduce the concept of \emph{reach}~\cite{federer1959curvature,cuevas2009set, 
aamari2019estimating, aamari2019nonasymptotic} (also known as condition number in \cite{niyogi2008finding} and minimal feature size in~\cite{chazal2005lambda}) to describe the smoothness of a manifold.
The reach is the longest distance away from $M$, in which every point within this distance to $M$ has a unique projection onto $M$, that is,
\begin{equation}
\reach(M) = \sup\{r>0: \forall x\in M\oplus r, \mbox{ $x$ has a unique projection onto $M$}\}.
\end{equation}
The reach can be viewed as the largest radius of a ball that can freely move along the manifold $M$. See figure~\ref{Fig::reach} for an example. 
The reach has been used in nonparametric set estimation 
as a condition to guarantee the stability of a set estimator \cite{chen2015asymptotic,chen2016density,cuevas2009set}.

\begin{figure}
\centering
	\subfigure[]
	{
		\includegraphics[scale=0.5]{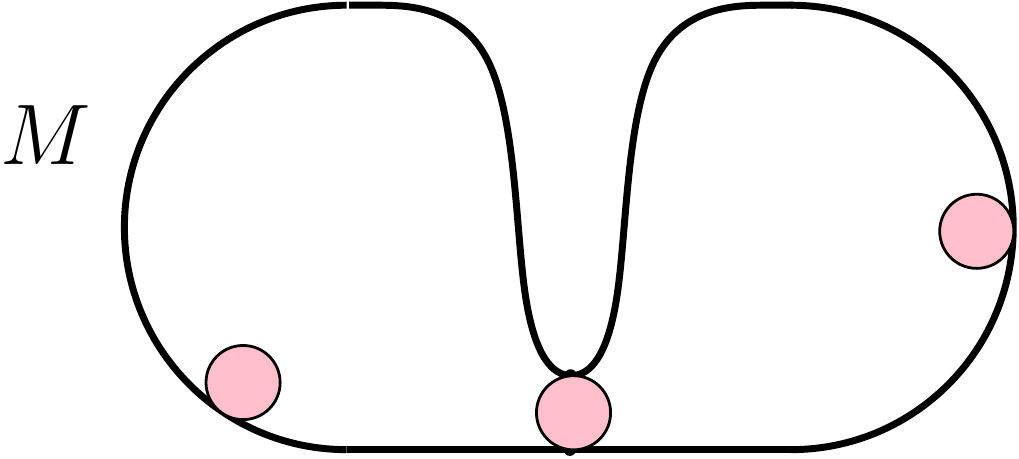}
	}
	\subfigure[]
	{
		\includegraphics[scale=0.5]{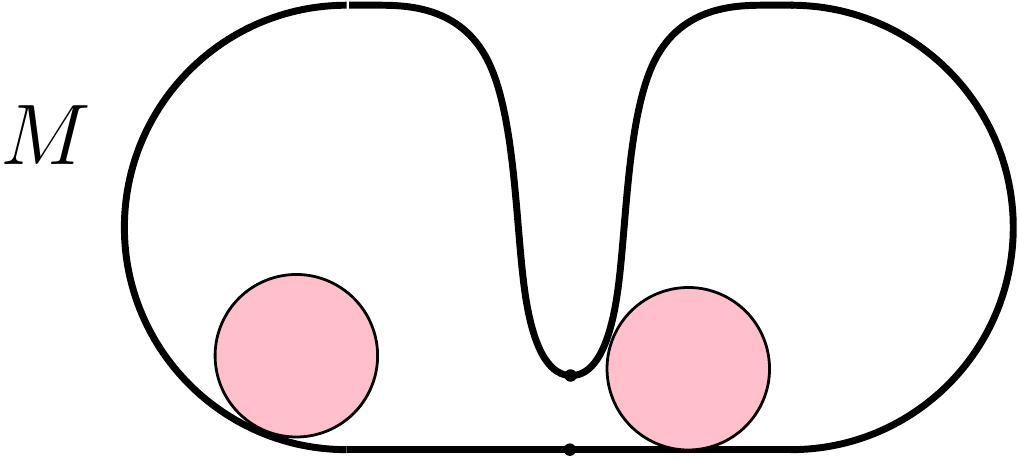} 
	}
\caption{An illustration for reach. Reach is the largest radius for a ball that can freely move along the manifold $M$ without penetrating any part of $M$. In (a), the radius of the pink ball is the same to the reach. In (b), the radius is too large; hence, it cannot pass the small gap on $M$.}
\label{Fig::reach}
\end{figure}

The smoothness of $\Psi$ does not suffice to guarantee the smoothness 
of a solution manifold. 
Consider the example of a density level set $\{x: p(x) = \lambda\}$
with a smooth density $p(x)$.
By construction, this level set is a solution manifold with $\Psi(x) = p(x) - \lambda$, a smooth function. 
Suppose that $p(x)$ has two modes and a saddle point $c$. 
If we choose $\lambda = p(c)$, the level set does not have a positive reach. 
See Figure \ref{Fig::Saddle} for an example.


\begin{figure}
\centering
\includegraphics[scale=0.5]{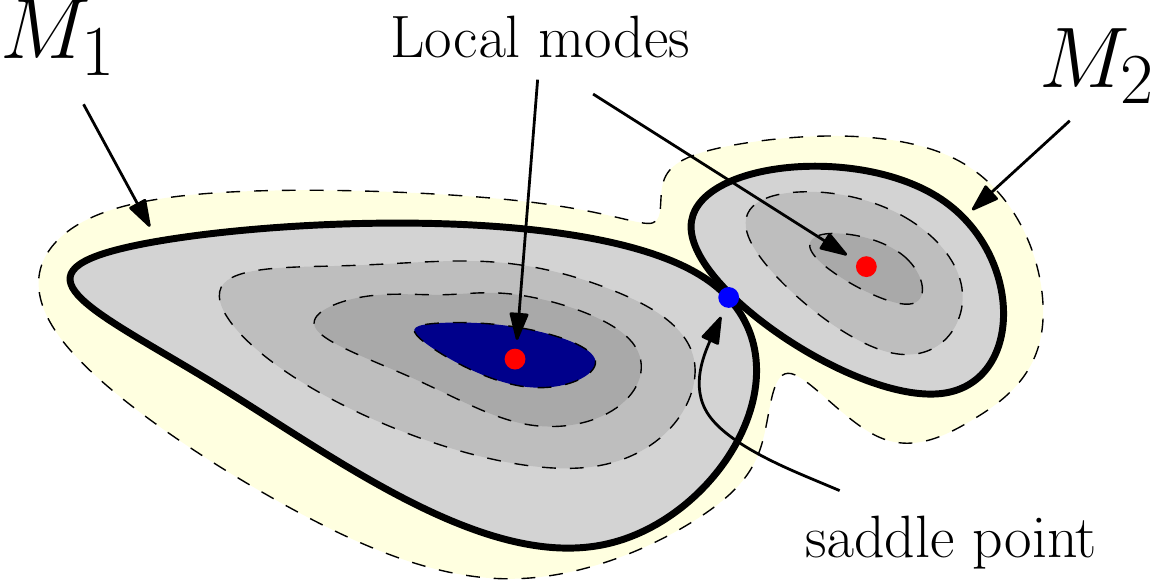}
\caption{An example of a smooth generator $\Psi$ but the resulting solution manifold $M=M_1\cup M_2$ may have $0$ reach. 
The dashed line is the contour lines for different levels.}
\label{Fig::Saddle}
\end{figure}
%


Although the smoothness of $\Psi$ is not enough
to guarantee a smooth $M$,
with an additional condition (F),
the solution manifold will be a smooth one as will be described in the following theorem.

\begin{theorem}	[Smoothness Theorem]
\label{thm::SM2} 
Conditions (D-2) and (F) imply that the reach of $M$ has a lower bound
$$
\reach(M) \geq \min \left\{\frac{\delta_0}{2},\frac{\lambda_0}{\norm{\Psi}^*_{\infty, 2}}\right\}.
$$
\end{theorem}
Theorem~\ref{Fig::Saddle} shows a lower bound on the reach
of a solution manifold. 
Essentially, it shows that as long as the generator function
is not flat around the solution manifold (assumption (F)), the resulting manifold
will be smooth. 
Note that for Theorem \ref{thm::SM2}, condition (D-2) 
can be relaxed to
a $2$-H{\"o}lder  condition
and the quantity $\norm{\Psi}^*_{\infty, 2}$ can be replaced by the corresponding H{\"o}lder's constant.

\begin{remark}
Reach is related to the curvature and a quantity called \emph{folding}~\cite{rice1967nonlinear}. 
Folding is defined as the smallest radius $r$ such that $B(x,r)\cap M$ 
is connected for each $x\in M$. 
The first quantity $\frac{\delta_0}{2}$ is related to the folding. 
The second quantity $\frac{\lambda_0}{K}$ is related to 
the curvature. When $s=1$, a similar result to Theorem~\ref{thm::SM2} appeared in \cite{Walther1997}.
Note that the reach is also related to the `rolling properties' and `$\alpha$-convexity'; see \cite{cuevas2009set} and Appendix A of \cite{pateiro2008set}.
\end{remark}

\subsection{Stability of Solution Manifolds}	\label{sec::MCT}
In this section, we show that when two generator functions are sufficiently close, the associated solution manifolds will be similar as well. 
We first start with a partial stability theorem, which holds under a weaker smoothness condition.

\begin{proposition}
Let $\Psi,\tilde{\Psi}:\mathbb{R}^{d}\mapsto \mathbb{R}^s$ be two vector-valued functions
and let
$$
M = \{x: \Psi(x) = 0\},\quad \tilde M = \{x: \tilde \Psi(x) = 0\}
$$
be the corresponding solution manifolds.
Moreover, let $\delta_0$ and $c_0$ be the constants in (F).
Assume that (D-1) holds within $M\oplus \delta_0$
and (F).
When $\norm{\tilde{\Psi}-\Psi}^*_{\infty, 0}<c_0$, 
we have
$$
\sup_{y\in \tilde M}d(y, M) = O\left(\|\tilde{\Psi}- \Psi\|^{(0)}_{\infty}\right).
$$

\label{prop::sta}
\end{proposition}

In the case of $s=d$, the conditions in Proposition \ref{prop::sta}
show many connections to the existing work. 
For instance, the convergence rate of estimating a mode (or a local mode)
is often based on a similar condition to (D-1); see, e.g., Theorem 2 of \cite{vieu1996note}. 
In the MLE and Z-estimation theory (see, e.g., Section~5.6 and 5.7 of \cite{vdv1998}), 
classical conditions often require the first order derivative of the score equation or estimating equations
to be uniformly bounded, which are similar to conditions of (D-1).
Moreover, in the MLE theory, we often need the derivative of the score function (or Fisher's information matrix
under appropriate conditions) to be non-singular within a small
neighborhood of the population maximizer.
This is exactly the same as condition (F). 
The constant $\lambda_0$ in (F) is the smallest absolute
eigenvalue of the derivative of the score function in the case of MLE problems.

Note that in the case of $s=d$, 
Proposition~\ref{prop::sta} is often enough for statistical consistency
because $M$ is a collection of disjoint points; thus, there is no need to consider the
smoothness of $M$. 
However, when $s<d$, the set $M$ contains infinite amount of points
and the smoothness of $M$ plays a role in terms of analyzing its
stability.
Therefore, we will need additional derivatives.

Before we formally state the stability theorem, we first introduce the concept of
the Hausdorff distance,
a common metric of sets.
The \emph{Hausdorff distance}
is defined as
\begin{align*}
\Haus(A,B) = \max\left\{\sup_{x\in B} d(x,A), \sup_{x\in A} d(x,B)\right\}.
\end{align*}
The Hausdorff distance is a distance between two sets and can be viewed as an $L_\infty$ type distance between sets.

\begin{theorem} [Stability Theorem]
\label{thm::LU0}
Let $\Psi,\tilde{\Psi}:\mathbb{R}^{d}\mapsto \mathbb{R}^s$ be two vector-valued functions
and let
$$
M = \{x: \Psi(x) = 0\},\quad \tilde M = \{x: \tilde \Psi(x) = 0\}
$$
be the corresponding solution manifolds.
Assume (D-2) and
(F) and that
$\tilde{\Psi}$ is bounded two-times differentiable.
When $\norm{\tilde{\Psi}-\Psi}^*_{\infty, 2}$ is sufficiently small,
\begin{description}
\item 1. (F) holds for $\tilde{\Psi}$.
\item 2. $\Haus(M,\tilde{M}) = O\left(\|\tilde{\Psi}- \Psi\|^{(0)}_{\infty}\right)$.
\item 3. 
$\reach(\tilde{M}) \geq \min\left\{\frac{\delta_0}{2},\frac{\lambda_0}{\|\Psi\|^*_{\infty,2}}\right\}+ O\left(\|\tilde{\Psi}- \Psi\|^*_{\infty,2}\right).$
\end{description}
\end{theorem}
Theorem \ref{thm::LU0} shows that two similar generator functions have similar solution manifolds. 
Claim 2 is a geometric convergence property indicating that
a consistent generator function estimator implies a consistent manifold estimator.
The need of a second-order derivative comes from the constants in (F). 
These constants are associated with  the second-order derivatives through Lemma~\ref{lem::relax}.
Claim 3 is the convergence in smoothness, which implies that $\tilde{M}$ cannot be very wiggly
when $\tilde{\Psi}$ is sufficiently closed to $\Psi$.



\begin{example}[Missing data]
The stability theorem (Theorem~\ref{thm::LU0}) provides a simple approach for obtaining the convergence rate of an estimator.
Consider the missing data example in Example~\ref{ex::missing}. 
The `population' solution manifold is the parameters $\theta = (\zeta_{x,y}, \mu_x, \xi)$
$$
\Theta = \{\theta: \Psi(\theta)  = 0 \}\subset \R^7
$$
such that $\Psi(\theta)\in\R^6$ is based on equation \eqref{eq::missing::psi}.
When we observed random samples of size $n$ in the form of $(X_i,Y_i,R_i=1)$ or $(X_i,R_i=0),$
we derived estimators $\hat P(X=x,Y=y,R=1)$ and $\hat P(X=x,R=0)$
and constructed an estimated version $\hat \Psi_n(\theta)$ by replacing $P(x,y,r=1)$ and $P(x,r=0)$ in equation \eqref{eq::missing::psi}
with the estimated versions. 
This leads to an estimated solution manifold 
$$
\hat \Theta_n = \{\theta:\hat \Psi_n(\theta)  = 0  \}\subset \R^d. 
$$
The stability theorem (Theorem~\ref{thm::LU0}) bounds the distance between $\hat \Theta_n$ and $\Theta$
via the difference $$
\max\{|\hat P(x,y,r=1)- P(x,y,r=1)|, |\hat P(x,r=0) - P(x,r=0)|: x,y=0,1\}.
$$
\end{example}


\section{Monte Carlo approximation to solution manifolds}	\label{sec::MCGD}

Given $\Psi$ or its estimator/approximation $\hat \Psi$, 
numerically finding the solution manifold $M $ is a non-trivial task. 
In this section, we propose a simple gradient descent procedure
to find a point on $M$. 
Note that though we describe the algorithm using $\Psi$,
we will apply the algorithm to $\hat \Psi$ in practice.
$M$ is the solution set of $\Psi$; thus, we may rewrite it as 
\begin{equation}
\begin{aligned}
M &= \{x: \Psi(x) = 0\}= \{x: f(x)= 0\},\\
f(x)&=\Psi(x)^T \Lambda \Psi(x) ,
\end{aligned}
\label{eq::criterion}
\end{equation}
where $\Lambda$ is an $s\times s$ positive definite matrix.
Let $x$ be an initial point. 
Consider the gradient flow $\pi_{x}(t)$:
$$
\pi_{x}(0) = x, \qquad\pi_{x}' (t) = -\nabla f(\pi_{x}(t)).
$$
Points in $M$ are stationary points of the gradient system;
moreover, they are the minima of the function $f(x)$.
Thus, we can use a gradient descent approach to find points on $M$.
Algorithm~\ref{alg::alg} summarizes the gradient descent procedure for approximating $M$.
Note that we may choose $\Lambda = \mathbb{I}$ to be the identity matrix.
In this case, $f(x) = \|\Psi(x)\|^2$, so we will be investigating the gradient descent flow
of $\|\Psi(x)\|^2$.
For the case of $d=s,$ this
is a common method in  numerical analysis to find the solution set of non-linear equations; see, e.g.,
Section 6.5 of  \cite{dennis1996numerical}.

\begin{algorithm}[tb]
\caption{Monte Carlo gradient descent algorithm} 
\label{alg::alg}
\begin{algorithmic}
\State 1. Randomly choose an initial point $x_0\sim Q$, where $Q$ is a distribution over the region of interest $\mathbb{K}. $
\State 2. Iterates 
\begin{equation}
x_{t+1}\leftarrow x_t - \gamma \nabla f(x_t)
\label{eq::GD}
\end{equation}
until convergence. Let $x_{\infty}$ be the convergent point. 
\State 3. If $\Psi(x_{\infty}) = 0$ (or sufficiently small), we keep $x_{\infty}$; otherwise, discard $x_{\infty}$.
\State 4. Repeat the above procedure until we obtain enough points for approximating $M$.
\end{algorithmic}
\end{algorithm}

Algorithm~\ref{alg::alg} consists of three steps:
a random initialization step,
a gradient descent step,
and a rejection step. 
The random initialization 
step allows us to explore different parts
of the manifold. 
The gradient descent step
moves the initial points to possible candidates on $M$
by iterating the gradient descent.
The rejection step ensures that
points being kept are indeed on the solution manifold. 

Note that the random initialization and rejection steps are popular strategies
in numerical analysis. They serve as a remedy to resolve the problem of
local minimizers of $f$ that is not in the solution manifold $M$. 
See the discussion in page 152 of \cite{dennis1996numerical}.

Figure~\ref{fig::ex02} shows an example 
of finding the solution manifold 
$$
\{(\mu,\sigma): P(-5<Y<2) = 0.5, Y\sim N(\mu,\sigma^2)\}
$$
using random initializations (from a uniform distribution over $[1,3]\times [2,4]$) 
and the gradient descent (we will provide more details 
on the implementations later in Example~\ref{ex::Gaussian}). 
We recover the underlying $1$-dimensional manifold structure using Algorithm~\ref{alg::alg}.

\subsection{Analysis of the gradient flow}

When an initial point is given, we perform gradient descent to find a minimum. 
We analyze this process by starting with an analysis of
the (continuous-time) gradient flow $\pi_x(t)$. 
The gradient descent algorithm can be viewed as a discrete-time approximation to the
continuous-time gradient flow.
To analyze the convergence of the flow and the algorithm,
we denote
$\Lambda_{\max}$ and $\Lambda_{\min}$ as the largest and smallest eigenvalues of $\Lambda$, respectively.

%


\begin{lemma}
Assume (D-2) and (F). Let $G_f(x) = \nabla f(x)$ and $H_f(x) = \nabla\nabla f(x)$
and $G_\psi (x) = \nabla \Psi(x)\in\R^{s\times d}$.
Then we have the following properties:
\begin{enumerate}
\item For each $x\in M$, 
\begin{enumerate}
\item the non-zero eigenvectors of $H_f(x)$ span the normal space of $M$ at $x$. 
\item the minimal non-zero eigenvalue 
$$
\lambda_{\min,>0}(H_f(x))\geq \psi_{\min}^2(x) \equiv \lambda^2_{\min, >0}(G_\psi(x)^T G_\psi(x))\Lambda_{\min}\geq 2\lambda_0^2\Lambda_{\min} .
$$
\item the minimal eigenvalue in the normal space of $M$ at $x$ $\lambda_{\min,\perp}(H_f(x))= \lambda_{\min,>0}(H_f(x)) $.
\end{enumerate}
\item Suppose that $x$ has a unique projection $x_M \in M$
and let $N_M(x)$ be the normal space of $M$ at $x_M$. 
If 
$d(x, M)<\delta_c = \min\left\{\delta_0, \frac{\Lambda_{\min}}{8d\Lambda_{\max}}\frac{\lambda_0^2 }{\|\Psi\|^*_{\infty,2} \|\Psi\|^*_{\infty,3}}\right\}$,
then 
$$
\lambda_{\min,\perp, M}(H_f(x)) \equiv \min_{v\in N_M(x_M)} \frac{v^T H_f(x) v}{\|v\|^2} = 
\min_{v\in N_M(x_M)} \frac{\|H_f(x) v\|}{\|v\|} \geq \lambda_0^2\Lambda_{\min}.
$$

\end{enumerate}

\label{lem::property2}
\end{lemma}

Property 1 in Lemma~\ref{lem::property2} describes the behavior of the Hessian $H_f(x)$ 
on the manifold. 
The eigenspace (corresponds to non-zero eigenvalues) is the same as the normal space of the manifold. 
With this insight, it is easy to understand property 1-(c)
that the minimal eigenvalue in the normal space is the same as the minimal non-zero eigenvalue. 
Property 2 is about the behavior of $H_f(x)$ around the manifold. 
The Hessian $H_f(x)$ is well-behaved as long as we are sufficiently close to $M$.
The following theorem characterizes several important properties of the gradient flow.

\begin{theorem}[Convergence of gradient flows]
Assume (D-3) and (F).
Let $\delta_c$ be defined in Lemma~\ref{lem::property2}.
Define $\pi_x(\infty) = \lim_{t\rightarrow\infty}\pi_x(t)$.
The gradient flow $\pi_x(t)$ satisfies the following properties:
\begin{enumerate}
\item 
(Convergence radius) For all $x\in M\oplus \delta_c$,
$
\pi_x(\infty) \in M.
$
\item (Terminal flow orientation) Let $v_x(t) = \frac{\pi_x'(t)}{\|\pi_x'(t)\|}$ be the orientation of the gradient flow.
	If $\pi_x(\infty)\in M$,
	then
	$v_x(\infty) = \lim_{t\rightarrow\infty} v_x(t)\perp M$ at $\pi_x(\infty)$. 
\end{enumerate}
\label{thm::GD}
\end{theorem}

The first result of Theorem~\ref{thm::GD} defines the convergence radius
of the gradient flow. 
The flow converges to the manifold
as long as the gradient flow starts within $\delta_c$ distance to the manifold. 
The second statement of the theorem characterizes the
orientation
of the gradient flow. The flow intersects the manifold
from the normal space of the manifold. Namely,
the flow hits the manifold orthogonally. 
If we choose $\Lambda = \mathbb{I}$ to be the identity matrix ($\Lambda_{\min}>0$ in this case), Theorem~\ref{thm::GD}
implies the convergence of the gradient flow of $\|\Psi\|^2$.

Note that Theorem~\ref{thm::GD} requires one additional derivative (D-3)
because we need to perform a Taylor expansion of the Jacobian of $\Psi$ around $M$
to ensure  that the gradient flow converges from
a normal direction to a manifold. 
We need the third-order derivatives to ensure that the remainders are small.

Suppose that  the initial point $x$ is drawn from the distribution $Q$, which has a PDF $q$, 
the convergent point $\pi_x(\infty)$
can be viewed as a random draw from a distribution $Q_M$ defined over the manifold $M$. 
The distribution $Q$ and the distribution $Q_M$ are associated via
the mapping induced by the gradient descent process;
thus,
$Q_M$ is a pushforward
measure of $Q$ \cite{bogachev2007measure}.
We now investigate how $Q$ and $Q_M$ are associated. 

For every point $z\in M$,
we define its basin of attraction \cite{chacon2015population, chen2016comprehensive} as
$$
A(z) = \{x: \pi_x(\infty)=  z\}.
$$
Namely, $A(z)$ is the collection of initial points that the gradient flow converges to $z\in M$. 
Let $\mathcal{A}_M = \cup_{z\in M} A(z)$ be the union of all basins of attraction.
The set $\mathcal{A}_M$
characterizes the regions where
the initialization leads to an accepted point in Algorithm~\ref{alg::alg}.
Thus, the acceptance probability of the rejection step of Algorithm~\ref{alg::alg}
is $Q(\mathcal{A}_M) = \int_{\mathcal{A}_M} Q(dx).$

Basin $A(z)$ has an interesting geometric property of
forming an $s$-dimensional manifold
under smoothness assumption. 
This result is similar to the stable manifold theorem in dynamical systems literature \cite{mcgehee1973stable,mcgehee1996new,banyaga2013lectures}.
In fact, it is more relevant to the local center manifold theorem (see, e.g., Section 2.12 of \cite{perko2013differential}).

\begin{theorem}[Local center manifold theorem]
Assume (D-3) and (F).
The basin of attraction $A(z)$
forms an $s$-dimensional manifold at each $z\in M$.
\label{thm::AMT}
\end{theorem}

An outcome from Theorem~\ref{thm::AMT}
is that the pushforward measure $Q_M$
has an $s$-dimensional Hausdorff density function \cite{mattila1999geometry,preiss1987geometry}
if $Q$ has a regular PDF $q$. 
Note that an  $s$-dimensional Hausdorff density at point $x$
is defined through
$$
\lim_{r\rightarrow 0}\frac{Q_M(B(x,r))}{C_s r^s},
$$
where $C_s$ is the $s$-dimensional volume of a unit ball. 
If the limit of the above equation exists, $Q_M$ has an $s$-dimensional Hausdorff density at point $x$.

Thus, if we obtain $Z_1,\cdots,Z_N\in M$ from
applying Algorithm~\ref{alg::alg},
we may view them as IID observations
from a distribution $Q_M$ defined over the manifold $M$,
and this distribution $Q_M$ has
an $s$-dimensional Hausdorff density function. 
The model that we observe IID $Z_1,\cdots,Z_N$
from a distribution supported over a lower-dimensional manifold
is common in computational geometry
\cite{Cheng2005,dey2006curve,dey2006provable,chazal2008smooth}.
Hence, Theorem~\ref{thm::AMT}
implies that Algorithm~\ref{alg::alg}
provides a new statistical example
for this model.

Note that \cite{arias2016estimation} proved the stability of a gradient ascent flow
when the target is to find the local modes of density function. 
The stability of basins of attraction was studied in \cite{chen2017statistical}
in a similar scenario.
These results may be generalized to solution manifolds.
The major technical issue that we need to solve is that
the convergent points of flows form a collection of infinite number of points.
Therefore, the analysis is much more complicated.
We leave this as a future work.

\subsection{Analysis of the gradient descent algorithm}

In Algorithm~\ref{alg::alg}, we did not perform the gradient descent using the flow $\pi_x$; instead, 
we used an iterative gradient descent approach that creates a sequence of discrete points $x_0,x_1,\cdots, x_\infty$ 
such that 
\begin{equation}
x_{t+1} = x_t -\gamma \nabla f(x_t),\quad x_0 = x,
\label{eq::GD::alg}
\end{equation}
where $\gamma>0$ is the step size. 

The gradient descent algorithm will diverge if  the step size $\gamma$ is chosen incorrectly.
Thus, it is crucial to investigate the range of $\gamma$
leading to a convergent point $x_{\infty}$
and how fast the sequence $\{x_t:t=0,1,\cdots\}$ converges to a point on $M$. 
The following theorem characterizes the algorithmic convergence along
with a feasible range of the step size $\gamma$.

\begin{theorem}[Linear convergence]
Assume (D-2) and (F).
When the initial point $x_0$ and step size $\gamma$
satisfy 
\begin{align*}
d(x_0, M)&< \delta_c = \min\left\{\delta_0, \frac{\Lambda_{\min}\lambda_0^2}{8d\Lambda_{\max} \|\Psi\|^*_{\infty, 2}\|\Psi\|^*_{\infty, 3}}\right\},\\
\gamma&<\min\left\{\frac{1}{\Lambda_{\max}\|\Psi\|^*_{\infty, 2}}, \frac{\Lambda_{\max}\|\Psi\|^*_{\infty, 2}}{4\lambda^4_0\Lambda^2_{\min}}, \delta_c\right\},
\end{align*}
we have the following properties for $t=0,1,2,\cdots$:
\begin{align*}
f(x_t)&\leq  f(x_0) \cdot \left(1- \gamma\frac{ \lambda^4_0\Lambda^2_{\min}}{\Lambda_{\max}\|\Psi\|^*_{\infty, 2}} \right)^t,\\
d(x_t, M)&\leq d(x_0, M) \left(1-\gamma\lambda_0^2 \Lambda_{\min}\right)^{\frac{t}{2}}.
\end{align*}
\label{thm::alg}

\end{theorem}


The convergence radius $\delta_c$ is the same as in Theorem~\ref{thm::GD}. 
Theorem~\ref{thm::alg} shows that
when the initial point is within the convergence radius of the gradient flow
and the step size is sufficiently small, 
the gradient descent algorithm converges linearly
to a point on the manifold. 
An equivalent statement of Theorem~\ref{thm::alg}
is that the algorithm takes only $O(\log(1/\epsilon))$ iterations to converge to $\epsilon$-error to the minimum.

The key step in the derivation of Theorem~\ref{thm::alg} 
is to investigate the minimal eigenvalue of the normal space $\lambda_{\min,\perp}(H_f(x))$
for each $x\in M$.
This quantity (appears in the theorem through the lower bound $\lambda_0^2\Lambda_{\min}$)
controls the flattest direction of $f(x)$ in the normal space. 
The three requirements on the step sizes are due to different reasons.
The first requirement  ($\gamma<\frac{1}{\Lambda_{\max}\|\Psi\|^*_{\infty, 2}}$) ensures that the objective
function is decreasing. 
The second requirement ($\gamma<\frac{\Lambda_{\max}\|\Psi\|^*_{\infty, 2}}{\lambda_0^2\Lambda_{\min}}$)
establishes the convergence rate. 
The third requirement ($\gamma <\delta_c$) guarantees that
the Hessian matrix behaves of $f$ is well-behaved when applying the gradient descent algorithm. 
The first and third requirements together are enough for the convergence of the gradient descent algorithm
but will not lead to the convergence rate. 
We need the additional second requirement to obtain the convergence rate.

Theorem~\ref{thm::alg} is a very interesting result. The function $f(x)$
is non-convex within any neighborhood of $M$ (i.e., not locally convex), but the gradient descent algorithm still 
converges linearly to a stationary point. 
An intuitive explanation of this result is that
the function $f(x)$ is a `directionally' convex function 
in the normal subspace of $M$ (Property 2 in Lemma~\ref{lem::property2}). 
Note that similar to Theorem~\ref{thm::GD},
Theorem~\ref{thm::alg} applies to the gradient descent algorithm with $\Lambda= \mathbb{I}$.


%

\section{Statistical Applications}	\label{sec::app}
In this section, we show that
the idea of solution manifolds can be applied to various statistical problems.
We also include a Bayesian approach that finds a credible region on a solution manifold
in Appendix \ref{sec::bayesian}.

\subsection{Missing data}	\label{sec::missing}

The solution manifold framework we developed can be
used to analyze the nonparametric bound in the missing data problem \cite{manski1990nonparametric,manski1999identification, chen2018monte}.
We use Example~\ref{eq::missing::psi} to illustrate the idea, but
our analysis can be generalized to a complex missing data scenario. 
The nonparametric bound refers to the feasible range of parameters $\theta = (\zeta_{x,y}, \mu_x, \xi: x,y=0,1)\in\R^7$
without any additional assumptions. 
Hence, the only constraint for these seven parameters is the six equations in equation \eqref{eq::missing::psi}. 
Thus, we know that the resulting parameter space is a one-dimensional manifold. 

This manifold will be a smooth one due to Theorem~\ref{thm::SM2}.
The stability theorem informs us that when we estimate the constraints by sample analogues,
the estimated manifold (can be viewed as an estimated nonparametric bound)
will be at $O_P(1/\sqrt{n})$ distance to the population manifold by the stability theorem (Theorem~\ref{thm::LU0}). 

Algorithm~\ref{alg::alg} provides a simple approach for numerically finding
points on this solution manifold. 
We can
obtain a point cloud approximation of the 1D manifold characterizing the nonparametric bound
of all the parameters with multiple random initializations.

\subsection{Algorithmic fairness}

We now revisit the algorithmic fairness problem in Example~\ref{ex::fair}.
We have shown that a simple method of generating
a test fair classifier $Q$ from $W,A$ is to sample from a Bernoulli random variable
with a parameter $q_{W,A} = P(Q=1|W,A)$
that satisfies equation \eqref{eq::fair2}. 
This leads to a solution manifold $\Theta = \{\theta = (q_{w,a}: w,a=0,1): \Psi(\theta) = 0\}$,
where $\Psi(\theta)$ is described in equation \eqref{eq::fair2}. 
By Theorem~\ref{thm::SM2}, the collection of feasible parameters
will be a smooth manifold. When we estimate the underlying constraint by
a random sample, the convergence rate (of manifolds) is described by Theorem~\ref{thm::LU0}.

In practice, finding $\Theta$ is often not the ultimate goal.
Our goal is to find a classifier that is test fair and has a good classification error \cite{hardt2016equality, corbett2017algorithmic}. 
A conventional approach of measuring classification accuracy is via a loss function $L(y,y')$ 
and we want to find the optimal $q^*_{w,a}\in \Theta$
such that 
$$
q^* = {\sf argmin}_{q\in \Theta}R(q) = \E_{Q\sim q}(L(Y, Q)).
$$
This is essentially a manifold constraint maximization/minimization problem. 
This problem also occurs in the constraint likelihood inference (see next section)
where we want to find the MLE under a solution manifold constraint. 
We will discuss a unified treatment of this manifold-constraint optimization problem in the next section
and
propose an algorithm for it (Algorithm~\ref{alg::alg::MLE}). 
While the algorithm is written in the form of finding the MLE,
one can easily adapt it to the test fairness problem. 


\subsection{Parametric model}	\label{sec::para}

One scenario that the solution manifolds will be 
useful is analyzing parametric models. 
We provide two different examples
showing how solution manifolds
can be used in parametric modeling.
Suppose that we observe IID observations $X_1,\cdots, X_n$
from some distribution $P$,
and we model the distribution using a parametric model $P_\theta$
and $\theta\in\Theta$. Let $p_\theta$ be the PDF/PMF of $P_\theta$
and let 
$$
\ell(\theta|X_1,\cdots,X_n) = \frac{1}{n}\sum_{i=1}^n \log p_\theta(X_i)
$$
be the log-likelihood function. 
Note that in Appendix~\ref{sec::bayesian},
we also provide
a Bayesian procedure that approximates the posterior distribution on a manifold (Algorithm~\ref{alg::alg::AMP}).

\subsubsection{Constrained MLE}	\label{sec::LRT}


In the likelihood inference, we may need to compute
the MLE under some constraints.  
One example
is the likelihood ratio test
when
the parametric space $\Theta_0$ under $H_0$
is generated by equality constraints.
Namely, 
$$
\Theta_0 = \{\theta\in \Theta: \Psi(\theta) = 0\}.
$$
This problem occurs in algebraic statistic and asymptotic theories 
can be found in \cite{michalek2016exponential} and Section 5 of \cite{drton2007algebraic}.

To use the likelihood ratio test, we need to find the MLEs under both $\Theta_0$ and $\Theta$.
Finding the MLE under $\Theta$ is a regular statistical problem. 
However, finding the MLE under $\Theta_0$ may not be easy
because of the constraint $\Psi(\theta)= 0 $. 
We propose to a procedure combining the gradient ascent of the likelihood function
and the gradient descent to the manifold
to compute the constrained MLE.
Algorithm~\ref{alg::alg::MLE} describes the procedure, and Figure~\ref{fig::ex02_04}
provides a graphical illustration.
This algorithm consists of a one-step gradient ascent of the likelihood function (Step 3)
and a gradient descent to manifold (Algorithm \ref{alg::alg}; steps 4 and 5).

The stopping criterion (Step 6) is
that $\nabla \ell(\theta^{(m)}_{\infty} |X_1,\cdots, X_n)$ belongs
to the row space of $\nabla\Psi(\theta^{(m)}_{\infty})$.
Due to Lemma~\ref{lem::normal},
the row space of $\nabla\Psi(\theta^{(m)}_{\infty})$
is the normal space of $M$ at $\theta^{(m)}_{\infty}$. 
It is easy to see that any critical points of the log-likelihood function 
on the manifold satisfy the condition that
the likelihood gradient belongs to the row space of $\nabla\Psi$;
thus, the constrained MLE is a stationary point in Algorithm~\ref{alg::alg::MLE}. 
As a result, we stop the algorithm when the stopping criterion occurs.
However, other local modes and saddle points and local minima are
also the stationary points. Hence, in practice, we need to
run the algorithm with multiple initial points to increase
the chance of finding the true MLE.

Note that one may replace the gradient ascent process by the EM algorithm.
However, 
the EM algorithm is not identical to a gradient ascent, so
it is unclear if the movement $\theta^{(m+1)}_{0} - \theta^{(m)}_{\infty}$
will be normal to the manifold $\Theta_0$ or not. 


\begin{algorithm}[tb]
\caption{Manifold-constraint maximizing algorithm} 
\label{alg::alg::MLE}
\begin{algorithmic}
\State 1. Randomly choose an initial point $\theta^{(0)}_0= \theta^{(0)}_\infty\in \Theta$.

\State 2. For $m=1,2,\cdots$, do step 3-6:
\State 3. {\bf Ascent of likelihood.} Update
\begin{equation}
\theta^{(m)}_{0} = \theta^{(m-1)}_{\infty} + \alpha \nabla \ell(\theta^{(m-1)}_{\infty} |X_1,\cdots, X_n),
\label{eq::MLE::GD}
\end{equation}
where $\alpha>0$ is the step size of the gradient ascent over likelihood function
and $\ell(\theta |X_1,\cdots, X_n)$ is the log-likelihood function.
\State 4. {\bf Descent to manifold.} For each $t=0,1,2,\cdots$ iterates 
$$
\theta^{(m)}_{t+1}\leftarrow \theta^{(m)}_t - \gamma \nabla f(\theta^{(m)}_t)
$$
until convergence. Let $\theta^{(m)}_{\infty}$ be the convergent point. 
\State 5. If $\Psi(\theta^{(m)}_{\infty}) = 0$ (or sufficiently small), we keep $\theta^{(m)}_{\infty}$; otherwise, discard $\theta^{(m)}_{\infty}$
and return to step 1.
\State 6. If $\nabla \ell(\theta^{(m)}_{\infty} |X_1,\cdots, X_n)$ belongs
to the row space of $\nabla\Psi(\theta^{(m)}_{\infty})$, we stop
and output $\theta^{(m)}_{\infty}$.


\end{algorithmic}
\end{algorithm}

\begin{example}[Testing a tail probability in a Gaussian model]	\label{ex::Gaussian}
To illustrate the idea, suppose that $X_i\in\R$ is from an unknown distribution that we place a parametric model on it.
We further assume that the parametric distribution is
a Gaussian $N(\mu,\sigma^2)$ with unknown mean and variance. 
Consider the null hypothesis
$$
H_0:  P(r_0\leq Y\leq r_1) = s_0
$$
for some given $s_0>0$ and $r_0,r_1$ (note that this example also appears in Figure~\ref{fig::ex02}).
Let  $\Phi(y) = P(Z\leq y)$ denote the CDF of a standard normal.
The null hypothesis $H_0$ forms the following constraint on $(\mu,\sigma^2)$:
$$
s_0 = \Phi\left(\frac{r_1-\mu}{\sigma}\right)-\Phi\left(\frac{r_0-\mu}{\sigma}\right).
$$
Thus,
$$
\Psi(\mu,\sigma) = \Phi\left(\frac{r_1-\mu}{\sigma}\right)-\Phi\left(\frac{r_0-\mu}{\sigma}\right) - s_0\in\R.
$$
The feasible set of $(\mu,\sigma^2)$ forms a $1D$ solution manifold in $\R^2$.
%
It is difficult to find the analytical form of the MLE under $H_0$,
but we may use the method in Algorithm~\ref{alg::alg::MLE}
to obtain a numerical approximation.
The derivative of $\Psi(\mu,\sigma)$ with respect to $\mu$ and $\sigma$
has the following closed-form:
\begin{align*}
\frac{\partial}{\partial \mu} \Psi(\mu,\sigma) &= -\frac{1}{\sigma} \phi\left(\frac{r_1-\mu}{\sigma}\right)+ \frac{1}{\sigma}\phi\left(\frac{r_0-\mu}{\sigma}\right) \\
\frac{\partial}{\partial \sigma} \Psi(\mu,\sigma)&= -\frac{r_1-\mu}{\sigma^2} \phi\left(\frac{r_1-\mu}{\sigma}\right)+\frac{r_0-\mu}{\sigma^2}\phi\left(\frac{r_0-\mu}{\sigma}\right),
\end{align*}
where $\phi(y) =\frac{1}{\sqrt{2\pi}}e^{-y^2/2}$  is the PDF of the standard normal.
Algorithm~\ref{alg::alg::MLE} can easily be implemented with the above derivatives.
Figure~\ref{fig::ex02_04} shows an example of applying Algorithm~\ref{alg::alg::MLE}
to this example with $r_1=-5,r_2=2$ and $s_0=0.5$ (and $1000$ random numbers generated from
$N(1.5,3^2)$). 
All  five random initial points converge to
the maximum on the manifold.
\end{example}

\begin{figure}
\center
\includegraphics[height=2in]{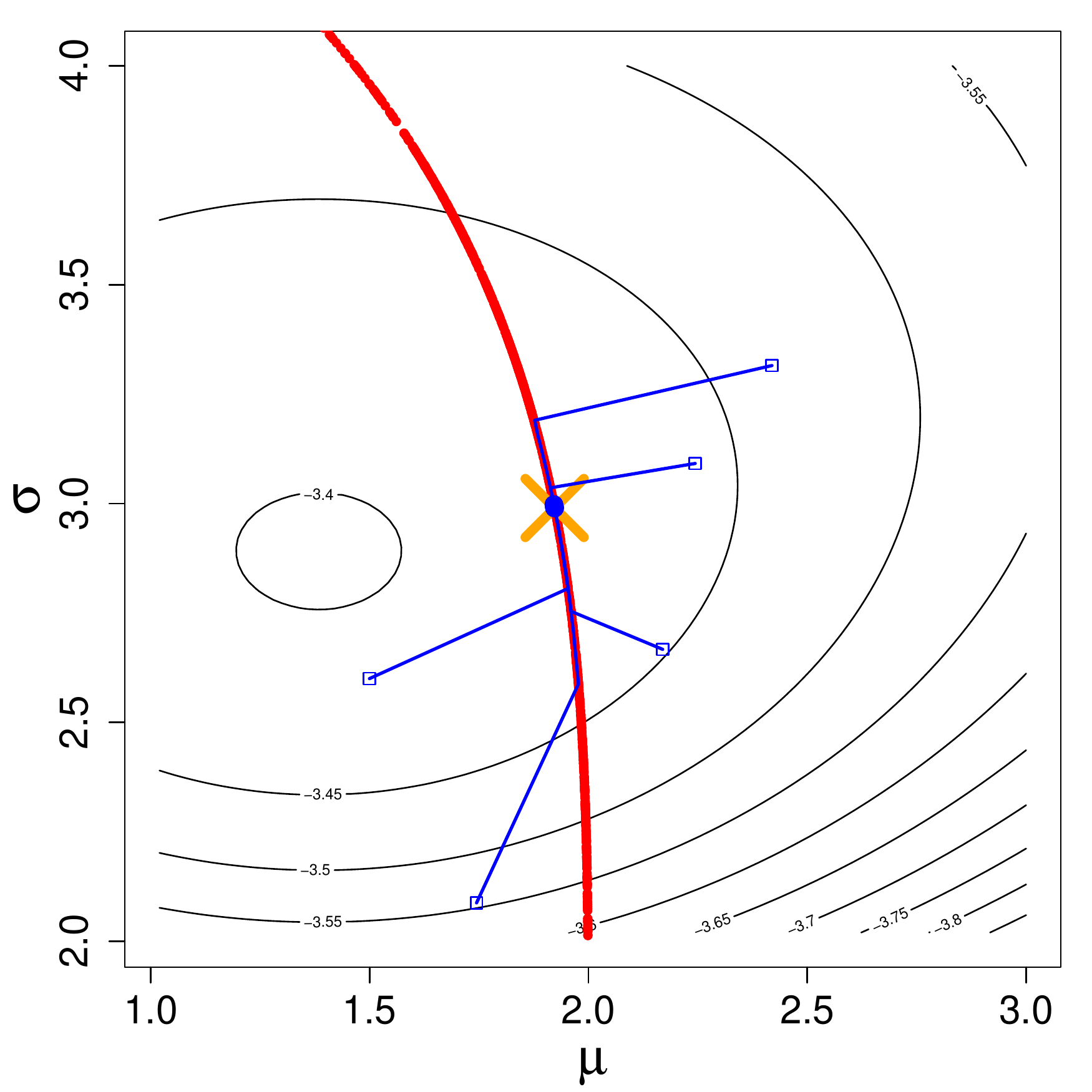}
\includegraphics[height=2in]{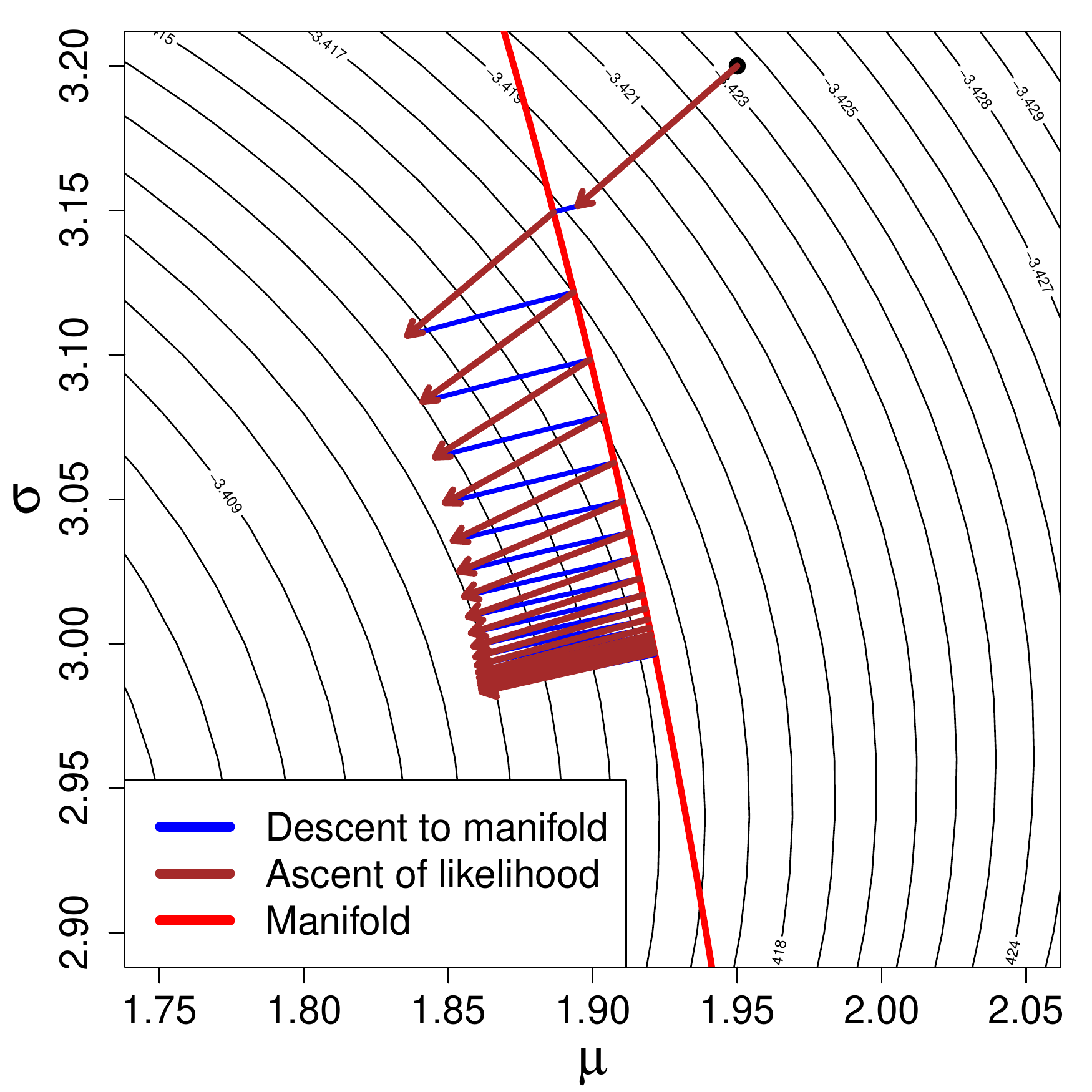}
\caption{An example illustrating how Algorithm \ref{alg::alg::MLE} works.
We consider the example of estimating the tail probability in a Gaussian model $N(\mu,\sigma^2)$
with the constraint $P(-5\leq X\leq 2) = 0.5$. 
We generate $n=1000$ points from $N(1.5,3^2)$ and display the log-likelihood function in the two panels (contours are the
log-likelihood surface).
{\bf Left:} We initialize Algorithm \ref{alg::alg::MLE} with five random points (blue boxes) 
and the algorithm creates an ascending path (blue lines) to the maximum point (orange cross).
{\bf Right:} We illustrate  Algorithm \ref{alg::alg::MLE} by showing points in each iteration in the algorithm
in a zoom-in area relative to the left panel. 
Starting from the  solid black point, we first perform a gradient ascent with respect to the log-likelihood function (brown arrow)
then apply Algorithm~\ref{alg::alg} to descent to the solution manifold.
We keep repeating this process until it converges. 
}
\label{fig::ex02_04}
\end{figure}

\subsubsection{Partial identification and generalized method of moments}	\label{sec::PI}

The solution manifolds appear in the partial identification problem \cite{manski2003partial}
in Econometrics. 
One example is the moment constraint problem \cite{chernozhukov2007estimation},
also known as the generalized method of moments \cite{hansen1982large, hansen1982generalized}.
In this case,
we want to estimate parameter $\theta\in \R^d$
that solves the moment equation
$
\E(g(Y;\theta)) = 0,
$
where
$g(y;\theta)\in \R^s$ is a vector-valued function,
and $X$ is a random variable denoting the observed data. 
When $s<d$, the solution set (also called an identified set in \cite{chernozhukov2007estimation}) $M = \{\theta: \E(f(Y;\theta)) = 0 \}$
forms a solution manifold. 

Thus, the smoothness theorem (Theorem~\ref{thm::SM2}) and the stability theorem (Theorem~\ref{thm::LU0})
can be applied to this case. 
When the estimator is obtained by the empirical moment equation
$
\hat M_n = \left\{\theta: \frac{1}{n}\sum_{i=1}^nf(Y_i;\theta) = 0 \right\},
$
Theorem~\ref{thm::LU0}
implies 
$
{\sf Haus}(\hat M_n, M) \overset{P}{\rightarrow} 0
$
when the empirical moments $ \frac{1}{n}\sum_{i=1}^nf(Y_i;\theta) $ 
and its derivatives with respect to $\theta$
converge to the population moments $\E(g(X;\theta)) $ 
and its derivatives, respectively.

In generalized method of moments,
a common approach of finding a solution to $\E(g(Y;\theta)) = 0$
is by
minimizing a criterion function $Q(\theta) =  \E(g(Y;\theta))^T \Lambda \E(g(Y;\theta))$, where $\Lambda$
is a positive definite matrix \cite{hansen1982generalized}.
This is identical to the function $f$ defined in equation \eqref{eq::criterion}. 
Thus, the analysis in Section~\ref{sec::MCGD} can
be used to study the minimization problem in the generalized method of moments.

\begin{remark}
In econometrics, a similar problem to the solution manifold
is the inequality constraint problem, which occurs when we replace the equality constraints with inequality constraints \cite{tamer2010partial,
romano2010inference}, i.e., 
$
\E(g_\ell(Y;\theta))\leq 0
$
for $\ell=1,\cdots, s$.
The goal is to find $\theta$ satisfying the above inequality constraint. 
A common approach to finding the feasible set is by defining an objective function
$$
Q(\theta) = \left|\sum_{\ell=1}^s[\E(g_\ell(Y;\theta))]_+\right|^2,\quad [y]_+ = \max(y,0)
$$
such that the feasible set is $\{\theta: Q(\theta)=0\}$.
The inequality constraint implies that $\{\theta: Q(\theta)=0\}$
may not form a lower-dimensional manifold but
a subset of the original parameter space.

A common estimator of $\{\theta: Q(\theta)=0\}$ is 
$$
\left\{\theta: \hat Q_n(\theta)\leq c_n\right\} ,\quad \hat Q_n(\theta) = \left|\sum_{\ell=1}^s\left[\frac{1}{n}\sum_{i=1}^ng_\ell(Y_i;\theta)\right]_+\right|^2
$$
for some sequence $c_n\rightarrow 0$. 
Note that by properly choosing $c_n$, we may construct both an estimator and a confidence region;
see \cite{chernozhukov2007estimation} and \cite{romano2010inference} for more details.

\label{rm::inequality}
\end{remark}

\subsection{Nonparametric set estimation}	\label{sec::NP}

\begin{figure}
\center
\includegraphics[height=1.5in]{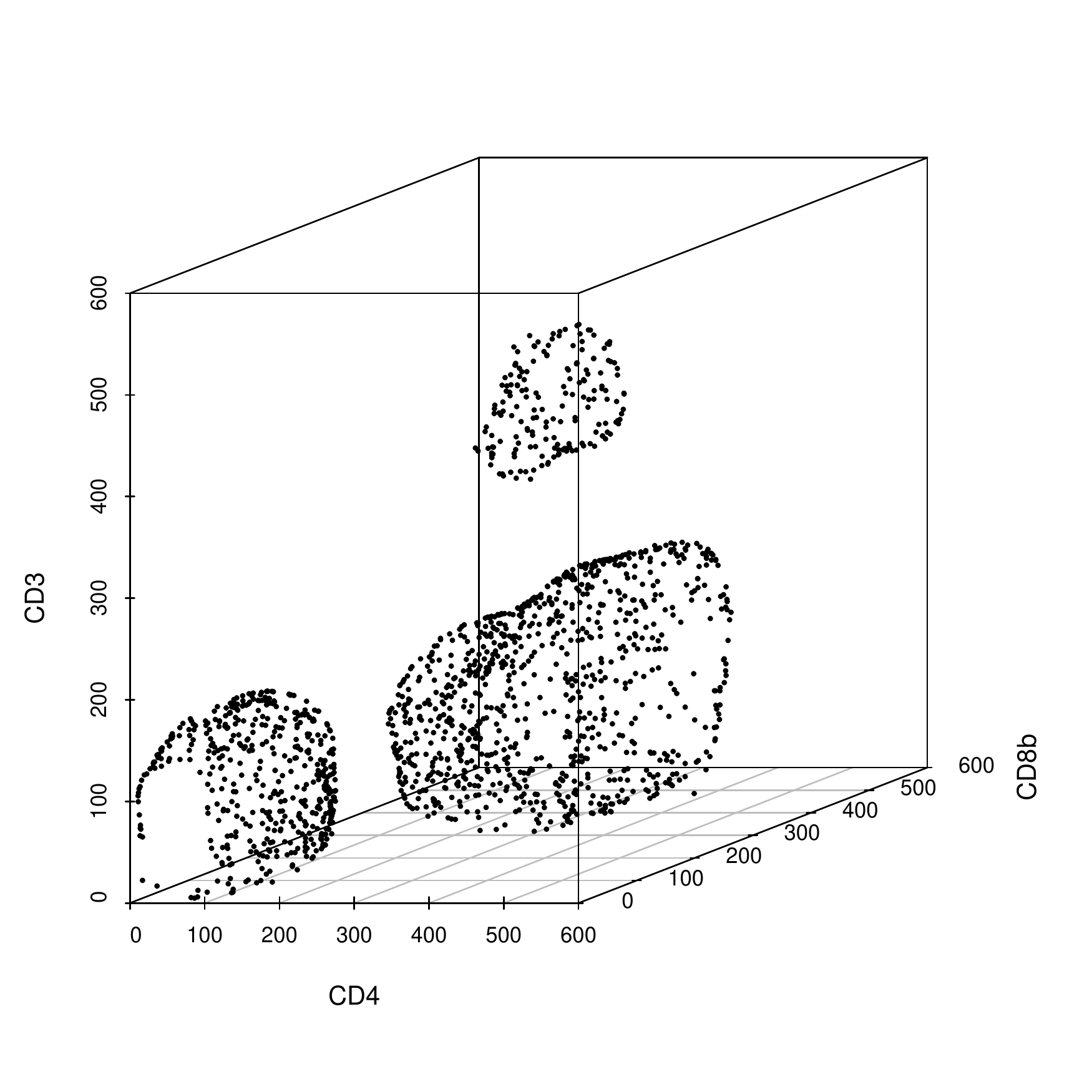}
\includegraphics[height=1.5in]{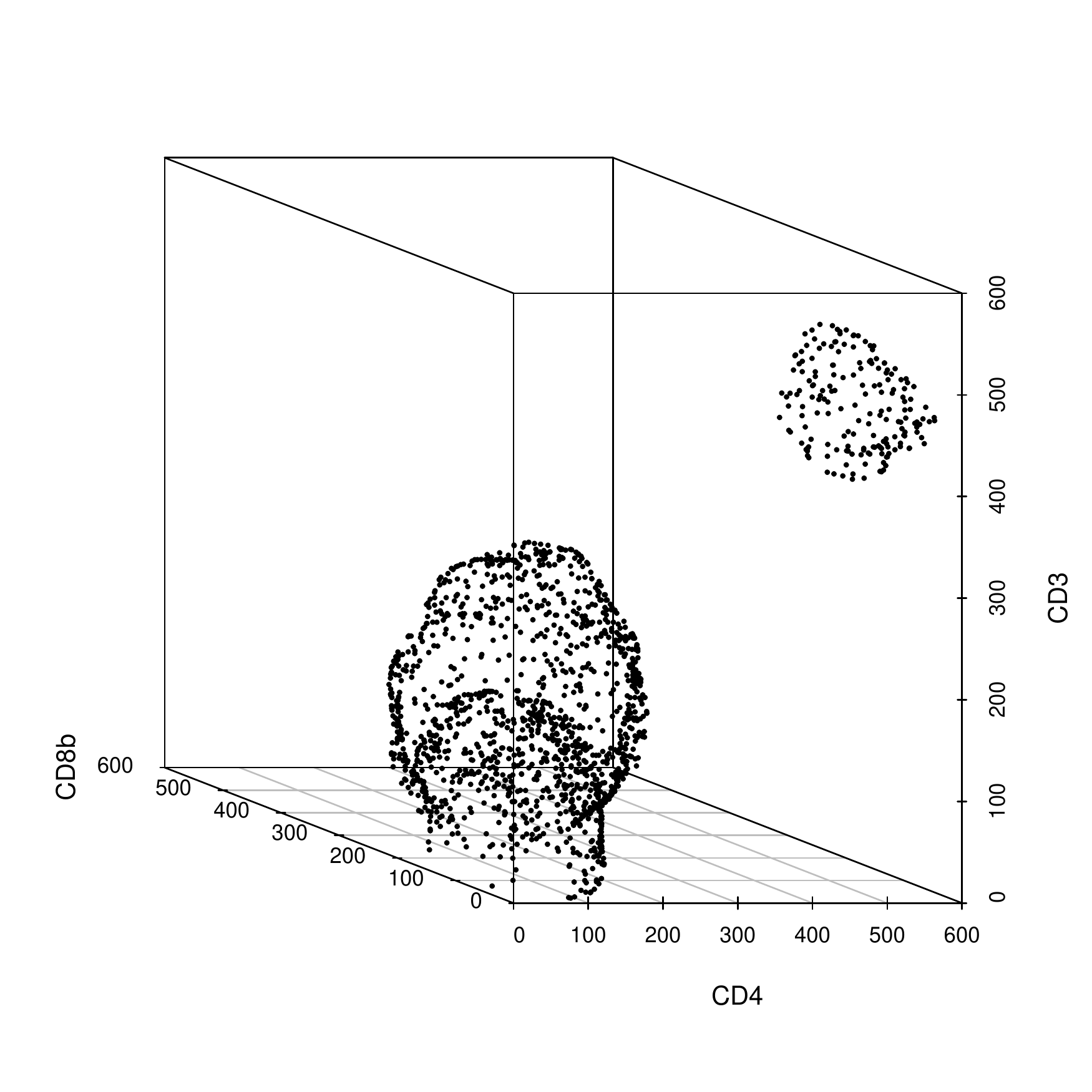}
\includegraphics[height=1.5in]{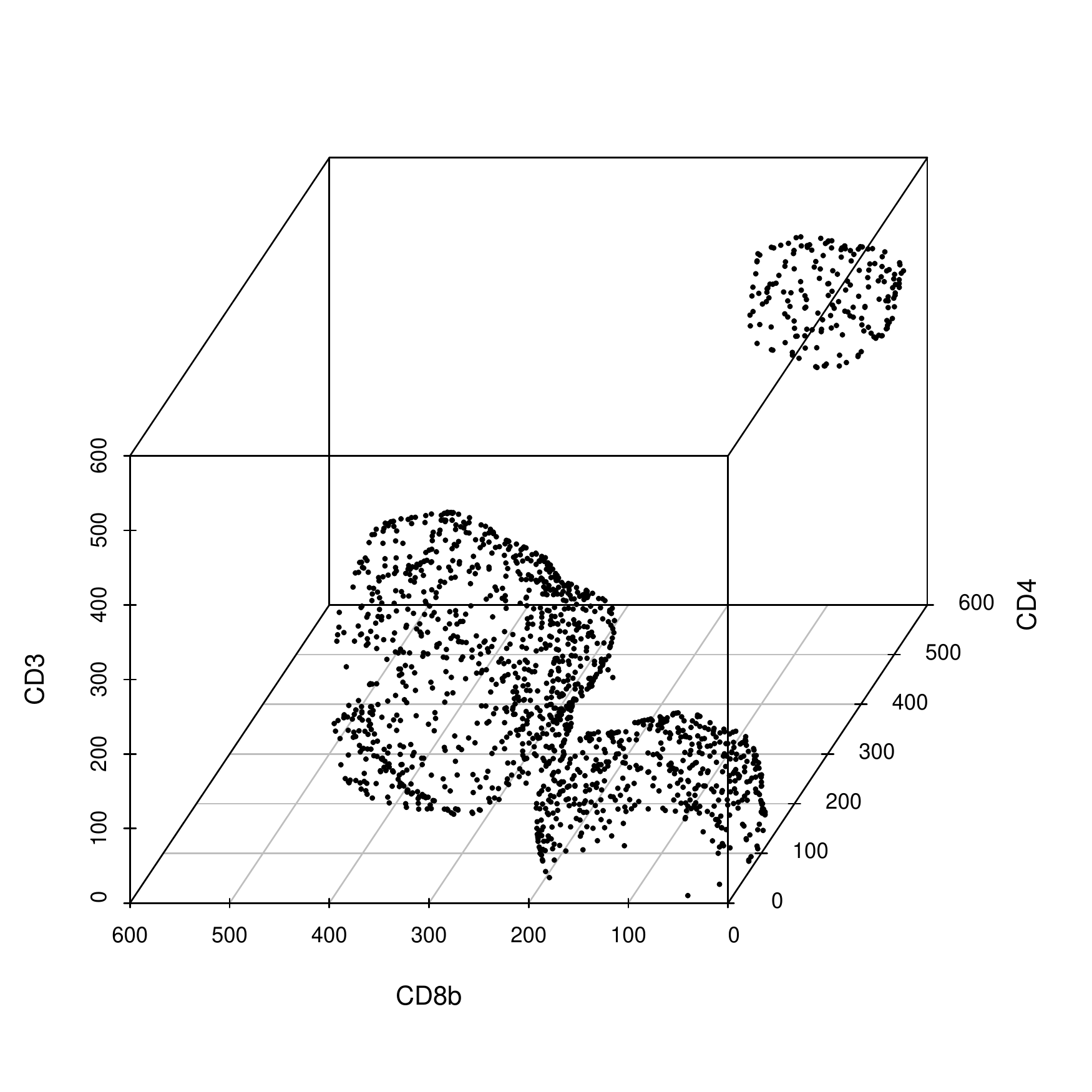}
\caption{An example of approximating a density level set.
This is the Graft-versus-Host Disease (GvHD) data that we borrowed from \texttt{mclust} package in \texttt{R}.
We use the control group and choose three variables: \texttt{CD3}, \texttt{CD4}, and \texttt{CD8b}. 
We apply a Gaussian kernel density estimator with bandwidth $h=20$ (same in all coordinates). 
The level set of interest is the density level corresponding to $25\%$ quantile of densities at all observations.
The three panels display the level set under three different angles. 
}
\label{fig::GvHD}
\end{figure}

Solution manifolds occur in many scenarios of nonparametric set estimation problems. 
One famous example is the density level set problem in which the parameter of 
interest is the (density) level set $\{x: p(x) = \lambda\}$,
and $p$ is the PDF that generates our data, and $\lambda$ is a pre-specified level. 
In this case, the smoothness theorem (Theorem~\ref{thm::SM2}) yields the same
result as \cite{chen2016density}.
Moreover, the stability theorem (Theorem~\ref{thm::LU0}) 
suggests that 
the convergence rate under the Hausdorff distance will
be the rate of estimating the density function,
which is consistent with several existing works \cite{Cadre2006,rinaldo2010generalized,rinaldo2012stability}. 

The methods developed in Section~\ref{sec::MCGD}
can be used to find points on the level set. 
As an illustration, Figure~\ref{fig::GvHD} shows an example
of approximating the level set using Algorithm~\ref{alg::alg}
with the Graft-versus-Host Disease (GvHD) data \cite{brinkman2007high} from \texttt{mclust} package in \texttt{R} \cite{baudry2010combining}. 
We use variables \texttt{CD3}, \texttt{CD4}, and \texttt{CD8b}
and focus on the control group.
The density is computed using a Gaussian kernel density estimator with an equal bandwidth $h=20$
in all coordinates. 
We choose the level as the $25\%$ quantile of all observations' densities. 
The three panels display the approximated level sets from
three angles. 
The three surfaces in the data indicate three connected components in the regions where the density is above this threshold.


In addition to the level set problem, 
the density ridges \cite{chen2015asymptotic,genovese2014nonparametric} are also examples of solution manifolds. 
The density $k$-ridges are the collection $\{x: V_k(x)^T \nabla p(x) = 0, \lambda_{k+1}(x)<0\}$,
where $V_k(x) = [v_{k+1}(x),\cdots, v_d(x)]$ denotes the matrix of $d-k$ eigenvectors of $\nabla\nabla p(x)$
corresponding to the smallest eigenvalues and $\lambda_k(x)$ is the $k$-th largest eigenvalue.
If we only pick the lowest $d-k$ eigenvectors,
we obtain a system of equations with $d-k$ equations, leading to a $k$-dimensional manifold
(under smoothness conditions). 
The stability theorem (Theorem~\ref{thm::LU0}) states that the convergence rate of a density ridge estimator will
be at the rate of estimating the Hessian matrix, which is consistent with the findings of \cite{genovese2014nonparametric}.

Note that the density local modes can be viewed as $0$-dimensional ridges.
In this special case, the matrix $V_1(x)$ is full rank under assumption (F);
thus,
$$
\{x: V_0(x)^T \nabla p(x) = 0, \lambda_{1}(x)<0\} = \{x:  \nabla p(x) = 0, \lambda_{1}(x)<0\}. 
$$
As a result, the function $\Psi$ that generates the local modes does not involve the Hessian matrix,
leading to the convergence rate of the mode estimator to be the same as the gradient estimation rate
rather than the rate of estimating the Hessian.


\section{Discussion}	\label{sec::discuss}

In this paper, we investigate both geometric and algorithmic properties of solution manifolds.
While the solution manifolds may seem to be abstract,
we  showed that they appear in various statistical problems including
missing data, algorithmic fairness, likelihood inference, and nonparametric set estimation. 
Hence, the methodologies and theories developed in this paper
provide a generic framework for analyzing all these problems.
This framework may inform us of the hidden relation among
all these seemingly different statistical problems. 
In what follows, we discuss some relevant topics to the solution manifolds.


%
%


\subsection{Smoothness, stability, and convergence of gradient flow}	\label{sec::compare}

We developed 5 major theoretical results:
smoothness theorem (Theorem~\ref{thm::SM2}), stability theorem (Theorem~\ref{thm::LU0}),
gradient flow theorem (Theorem~\ref{thm::GD}), local center manifold theorem (Theorem~\ref{thm::AMT}),
and algorithmic convergence theorem (Theorem~\ref{thm::alg}).
These results characterize different properties of  solution manifolds and
are often studied in various 
fields.
In our work, we showed that they all rely on a similar set of assumptions:
(D), the smoothness of $\Psi$, and (F), the curvature assumption of $\Psi$ around $M$.

Assumption (D) is more than enough for some theoretical results.
The smoothness theorem can be relaxed to assuming that $\Psi$
satisfies $\beta$-H{\"o}lder condition with various $\beta$.
For instance,
the (D-1) condition in
the weak stability result (Proposition~\ref{prop::sta}) can be relaxed to the
$1$-H{\"o}lder condition.
The condition (D-2) in the smoothness theorem (Theorem~\ref{thm::SM2}) can be replaced by the 
$2$-H{\"o}lder condition.

Moreover,
we observe
a hierarchy of smoothness corresponding to different theoretical results.
If we only have (D-1), the first order derivatives, we  have a weak stability result from Proposition~\ref{prop::sta}. 
If we make a further assumption (D-2),
we have a stability theorem (Theorem~\ref{thm::LU0}), a characterization of smoothness (Theorem~\ref{thm::SM2}),
and an algorithmic convergence (Theorem~\ref{thm::alg}).
Under an additional assumption (D-3),
we can derive an even stronger result of the corresponding gradient flow (Theorem~\ref{thm::GD} and \ref{thm::AMT})



%

\subsection{Connections to other fields}
We would like to point out that the results of this paper have
several connections to other fields.

\begin{itemize}
\item {\bf Econometrics.}
Solution manifolds occur in the partial identification problem (Section~\ref{sec::PI});
hence, our analysis provides some insights into the moment equality constraint problem \cite{chernozhukov2007estimation}. 
Our analysis on the gradient descent (e.g., Theorem~\ref{thm::GD})
can be applied to investigate the property
of the minimization problem in the generalized method of moments approach \cite{hansen1982large, hansen1982generalized}.

\item {\bf Dynamical systems.}
As mentioned before, Theorem~\ref{thm::AMT} is
related to the stable manifold theorem and the local center manifold theorem in 
dynamical systems \cite{mcgehee1973stable,mcgehee1996new,banyaga2013lectures,perko2013differential}. 
Our analysis provides statistical examples that
these theorems may be useful in data analysis.

%
%

\item {\bf Computational geometry.}
If we stop the gradient descent process early, 
we do not obtain points that are on the manifold. 
The resulting points $Z_1,\cdots, Z_n$ may be viewed as from
$Z_i  = X_i+\epsilon_i$, where $X_i\in M$ is from a distribution over the manifold
and $\epsilon_i$ is some additive noise. 
This model is a common additive noise model in the computational geometry literature
\cite{Cheng2005,cheng2005manifold,dey2006curve,dey2006provable,chazal2008smooth,boissonnat2014manifold}.
Our proposed method provides another concrete example of the manifold additive noise model.
\item {\bf Optimization.}
In general, a gradient descent method has a linear convergence rate
when the objective function is strongly convex and has a smooth gradient \cite{boyd2004convex,nesterov2018lectures}.
However,  in our setting,
the objective function $f(x)$ is non-convex (and is not locally convex),
but the gradient descent algorithm still
obtains a linear (algorithmic) convergence rate (Theorem~\ref{thm::GD}). 
This  reveals a class of non-convex objective functions
that can be minimized quickly using a gradient descent algorithm.

\end{itemize}

\subsection{Future work}

The framework developed in this paper has many potentials in other problems. 
We provide some possible directions that we plan to pursue in the future. 

\begin{itemize}

\item {\bf Log-linear model.}
The log-linear model of categorical variables is 
an interesting example
in the sense that it can be expressed as a solution manifold when there are constraints like conditional independence
but it may be unnecessary to use the developed techniques.
Consider a $d$-dimensional categorical random vector $X$ that takes values in $\{0,1,2,\cdots, J-1\}^d$.
The joint PMF of $X$ $p(x_1,\cdots,x_d)$ has $J^d$ entries with the constraint that $\sum_{x} p(x_1,\cdots,x_d) = 1$,
so it has $J^d-1$ degrees of freedom.
In the log-linear model, we reparametrize the PMF using the log-linear expansion:
$\log p(x) = \sum_{A} \psi_A(x_A)$, where $A $ is any non-empty subset of $\{1,2,\cdots, J\}$ and $x_A = (x_j: j\in A)$
with the constraint that $\psi_A(x_A)=0$ if any $x_j = 0$ for $j\in A$. 
Under the log-linear model, 
we reparametrize the joint PMF using the parameters $\Theta_{\sf  LL} = \{\psi_A(x_A): A\subset  \{1,2,\cdots, J\}, x_A\in\{0,1\}^{|A|}\}$,
where $|A|$ is the cardinality of $A$.
The feasible parameters in $\Theta_{\sf  LL} $ forms a solution manifold due to the aforementioned constraints.
However, 
common constraints in the log-linear model are that interaction terms $\psi_A=0$ for some $A$.
This leads to a flat manifold; hence, there is no need to use the developed technique.
We may need to use techniques from the solution manifold when the constraint is placed on the PMF $p(x_1,\cdots, x_d)$ rather than the log-linear models
because the constraints on the PMF lead to an implicit constraint on $\Theta_{\sf  LL} $. 
We leave this as future work.

\item {\bf Confidence regions of solution manifolds.}
Another future direction 
is to develop a method
for constructing the confidence regions of solution manifolds. 
There are two common approaches
to construct a confidence region of a set. 
The first one is based on the ``vertical uncertainty'', which is the uncertainty due to $\hat \Psi- \Psi$. 
This idea has been applied in generalized method of moment problems \cite{chernozhukov2007estimation, romano2010inference, chen2018monte}
and
level set estimations \cite{mammen2013confidence,qiao2019nonparametric,cheng2019nonparametric}
The other approach is based on the ``horizontal uncertainty'', which is the uncertainty due to ${\sf Haus}(\hat M, M)$. 
This technique has been used in constructing confidence sets of density ridges and level sets \cite{chen2015asymptotic,chen2016density}. 
Based on these results, we believe that
it is possible to develop a procedure for constructing
the confidence regions of solution manifolds.
We leave this as future work. 


\item {\bf A new class of non-convex problems.}
We observe an interesting phenomenon in Theorem~\ref{thm::alg}. Although the objective function $f(x) = \Psi^T \Lambda \Psi(x)$
is non-convex, we still obtain a linear (algorithmic) convergence.
Note that for a non-convex but locally convex around the minimizer,
the linear convergence can be established via assuming a local strong convexity of the objective function \cite{balakrishnan2017statistical}, i.e.,
$f(x)$ is strongly convex within $B(x^*, r)$ for some radius $r>0$ and $x^*$ is  the global minimizer.
However, our problem is more complicated in the sense that $f(x)$ is flat along $M$,
so it is not locally strongly convex. 
The key element in our result is assumption (F) stating that
$f(x)$ behaves like being ``locally strongly-convex'' in the normal direction of $M$.
Thus, with some additional structure on the non-convex function, we may still obtain a fast convergence. 
We will investigate how this may be useful in other non-convex optimization problems. 
In addition, the
analysis may be applied to other forms of $f(x)$ that are not limited to a ``squared''-type 
transformation of $\Psi(x)$ ($f(x)$ behaves like the square of $\Psi$), which may further improve the convergence rate. 
For instance, the gradient descent over $f_1(x) = \|\Psi(x)\|_1$  may also converge faster than over the function $f(x)$. 
We will investigate this in the future.

\end{itemize}


\appendix

\section{Bayesian inference}	\label{sec::bayesian}

The techniques we developed for solution manifolds
can be used for the Bayesian inference 
after some modifications.
One example is the univariate Gaussian with unknown mean $\mu$ and variance $\sigma^2$,
and a second moment constraint.
The parameter space is $\Theta(s_0) = \{(\mu,\sigma^2): \E(Y^2)=  \mu^2+\sigma^2 = s^2_0\}$. 
We 
place a prior $\pi(\theta)$ over $\Theta(s_0)$ 
that reflects our prior belief about the parameter $\theta = (\mu,\sigma)$.
However, 
how to sample from $\pi$ (and the posterior) is a non-trivial task
because $\pi$ is supported on a manifold.

The Monte Carlo approximation method in Section~\ref{sec::MCGD}
offers a solution to sampling from $\pi$. 
With a little modification of Algorithm~\ref{alg::alg},
we can approximate the posterior distribution defined on the solution manifold. 
Let $\pi$ be a prior PDF defined over the solution set $M = \{\theta: \Psi(\theta)= 0 \}$
where $\Psi:\Theta\mapsto\R^k$.
We observe IID observations $X_1,\cdots, X_n$
that are assumed to be from a parametric model $p(x|\theta)$. 
The posterior distribution of $\theta$
will be 
$$
\pi(\theta|X_1,\cdots,X_n)\propto
\begin{cases}
\pi(\theta) \prod_{i=1}^n p(X_i|\theta),&\quad \mbox{if $\theta \in M$;}\\
0,&\quad \mbox{if $\theta \notin M$.}
\end{cases}
$$
We propose a method that approximates the posterior distribution
using a weighted point cloud. 
Our approach is formally described
in Algorithm~\ref{alg::alg::AMP}. 
Note that the algorithm we develop
only requires the ability to evaluate a function $\rho(\theta) \propto \pi(\theta)$.
We do not need the exact value of the prior density.

\begin{example}[Bayesian analysis of Example~\ref{ex::Gaussian}]
Figure~\ref{fig::ex02_05} shows an example
of 90\% credible intervals and the MAPs
under three scenarios: prior distribution only (left panel),
posterior distribution with $n=100$ (middle panel),
and posterior distribution with $n=1000$ (right panel). 
This is the same setting in Example~\ref{ex::Gaussian} and Figure~\ref{fig::ex02_04},
where the manifold is formed by the constraint $P(-5<X<2) = 0.5$ with $X\sim N(\mu,\sigma)$.
We choose the prior distribution (density) as 
$$
\pi(\mu,\sigma)\propto \phi(\mu; 2,0.2)\phi(\sigma;2.5, 0.2) I((\mu,\sigma)\in M),
$$
where $\phi(x;a,b)$ is the density of $N(a,b^2)$.
In the left panel, the credible interval is completely determined by the prior distribution
and the MAP is the mode of the prior. 
In the middle and right panels, the data are incorporated into the posterior distributions.
Both the credible intervals and MAPs are changing because of the influence of the data.
Our method (Algorithm~\ref{alg::alg::AMP} and equation \eqref{eq::CI1}) provides a simple and elegant approach 
of approximating the credible intervals 
on the manifold.
\end{example}

\begin{algorithm}[tb]
\caption{Approximated manifold posterior algorithm} 
\label{alg::alg::AMP}
\begin{algorithmic}
\State 1. Apply Algorithm~\ref{alg::alg} 
to generate many points $Z_1,\cdots, Z_N\in M$. 

\State 2. Estimate a density score of $Z_i$
using 
$$
\hat\rho_{i,N} = \frac{1}{N} \sum_{j=1}^N K\left(\frac{\|Z_i-Z_j\|}{h}\right),
$$
where $h>0$ is a tuning parameter and $K$ is a smooth function such as a Gaussian.

\State 3. Compute the posterior density score of $Z_i$ as
\begin{equation}
\hat \omega_{i,N} = \frac{1}{\hat\rho_{i,N}}\cdot \hat \pi_{i,N},\quad \hat\pi_{i,N} =  \pi(Z_i)\cdot  \prod_{j=1}^n p(X_j|Z_i),
\label{eq::PDS}
\end{equation}
\Return Weighted point clouds $(Z_1,\hat \omega_{i,N}),\cdots, (Z_N,\hat \omega_{N,N})$.

\end{algorithmic}
\end{algorithm}

\begin{figure}
\center
\includegraphics[height=1.5in]{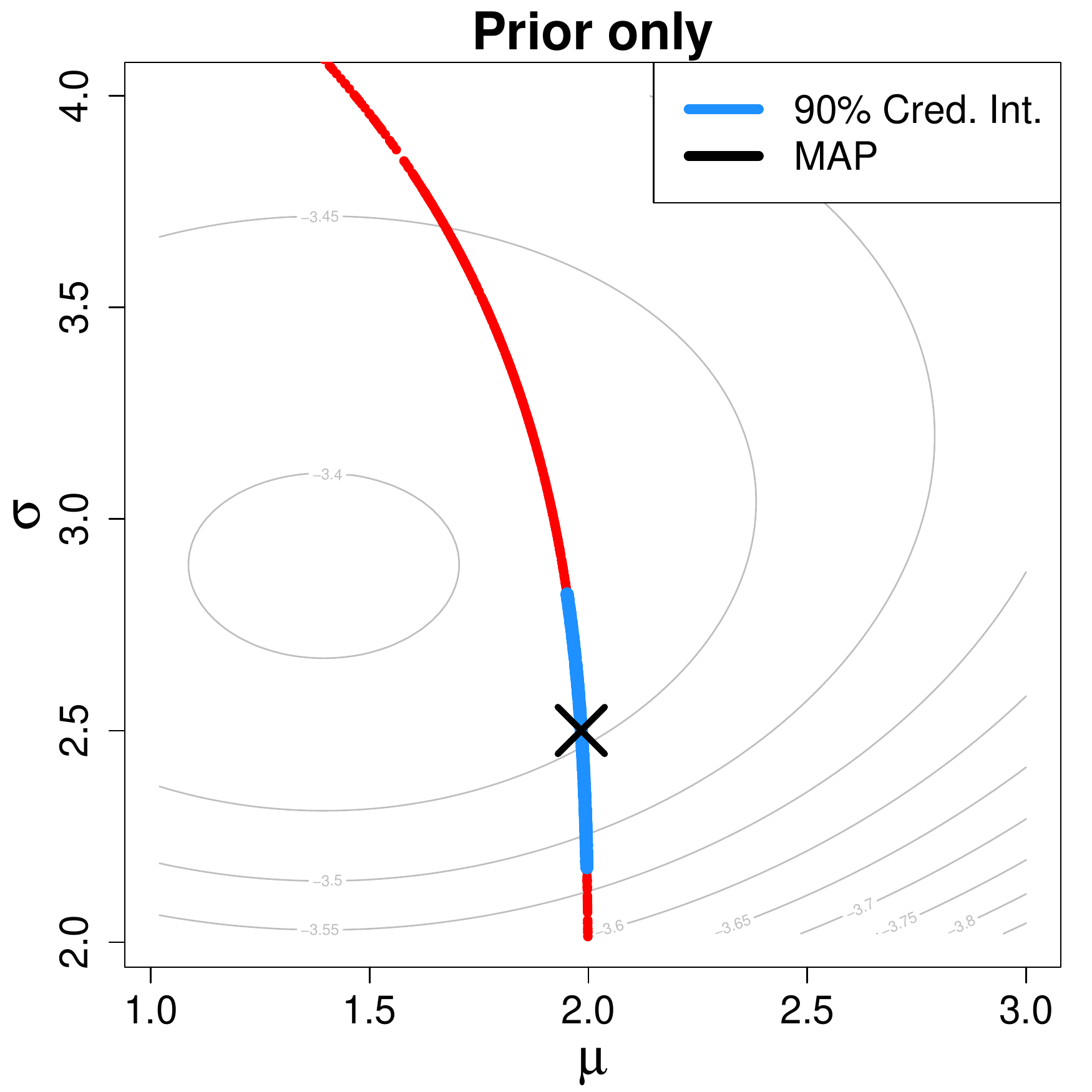}
\includegraphics[height=1.5in]{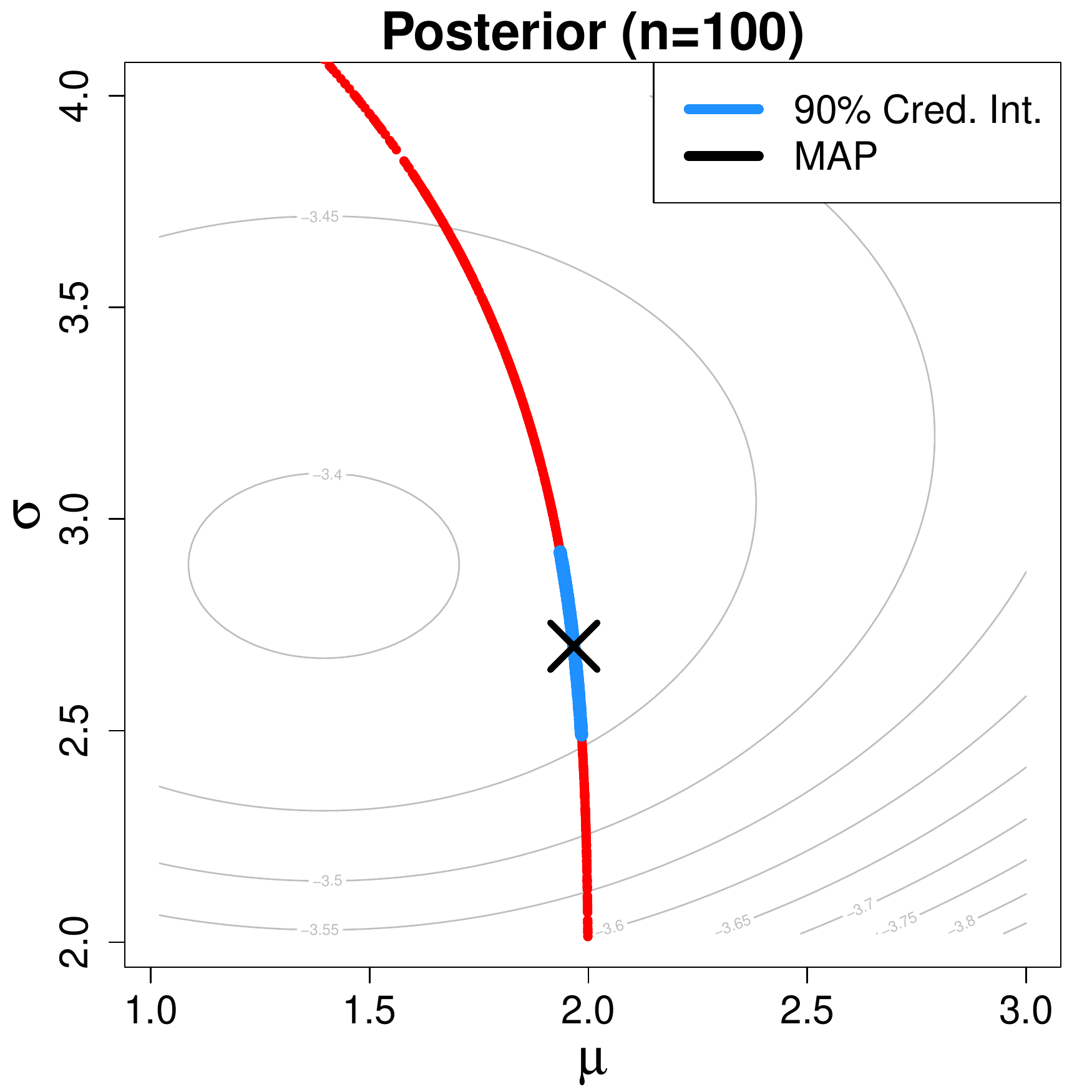}
\includegraphics[height=1.5in]{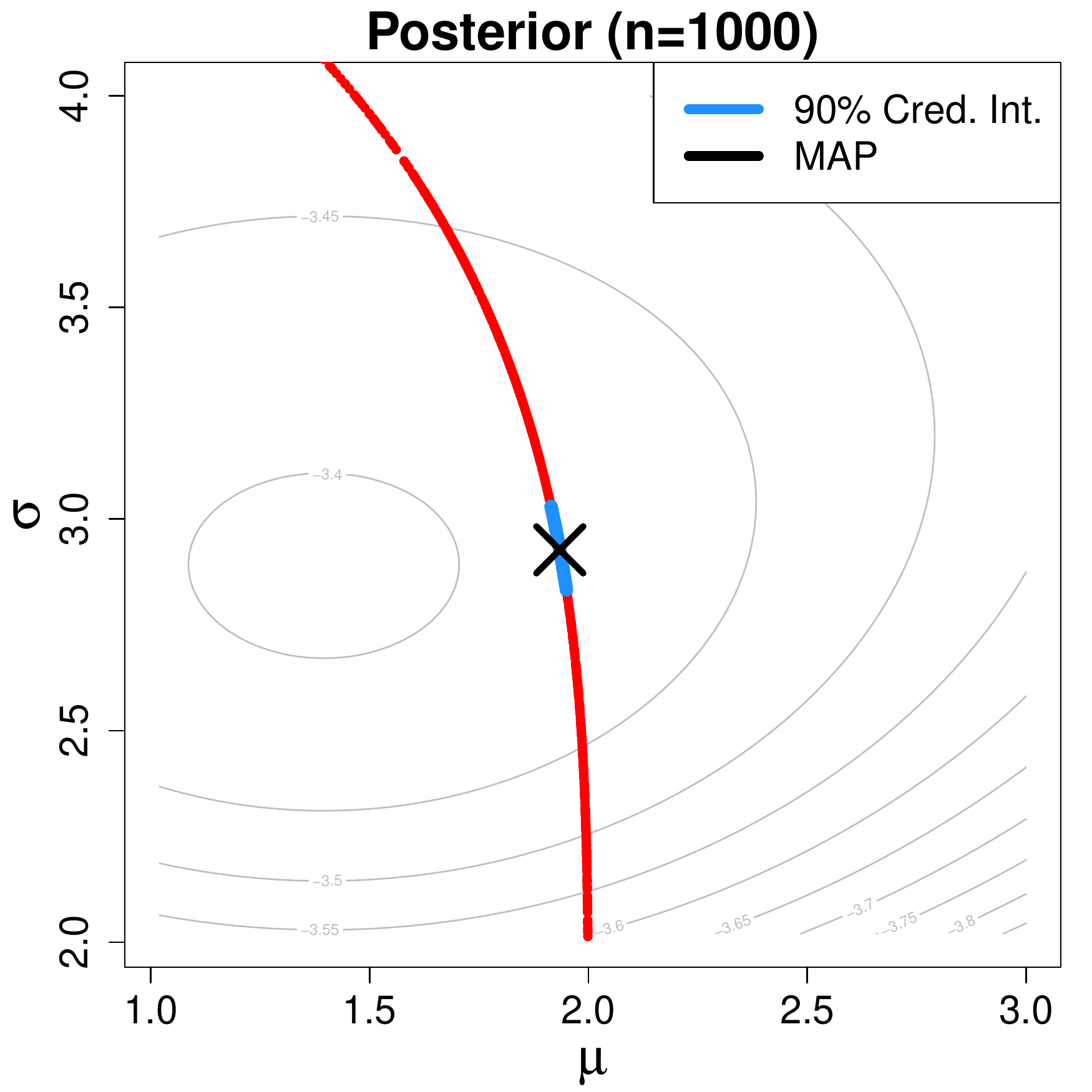}
\caption{An example showing the credible interval (credible region) and MAP
on the manifold. We use the same example in Figure \ref{fig::ex02_04}
and the prior distribution $\pi(\mu,\sigma)\propto \phi(\mu; 2,0.2)\phi(\sigma;2.5, 0.2) I((\mu,\sigma)\in M)$,
where $\phi(x;a,b)$ is the density of $N(a,b^2)$.
{\bf Left:}  90\% credible interval along with the MAP using only the prior distribution.
{\bf Middle:} we randomly generate $n=100$ observations from $N(1.5, 3^2)$ and compute
90\% credible interval and MAP from the posterior distribution.
{\bf Right:} the same analysis as the middle panel but now we use a sample of size $n=1000$.
Note that the background gray contours show the log-likelihood function (as an indication of how the likelihood function
will influence posterior).
}
\label{fig::ex02_05}
\end{figure}

To see why the outputs from Algorithm~\ref{alg::alg::AMP}
are a valid
approximation to the posterior density, 
note that the density score $\hat\rho_{i,N} $
is proportional to the underlying density of $Z_1,\cdots, Z_M$
defined over $M$.
Hence, the weighted point cloud $(Z_1,\hat\rho_{i,N}^{-1}),\cdots, (Z_N,\hat\rho_{N,N}^{-1})$
behaves like a uniform sample over $M$. 
Thus, to account for the unweighted point cloud density,
we have to rescale the posterior score of $Z_i$ in equation \eqref{eq::PDS}
by the factor $\hat\rho_{i,N}^{-1}$.
Note that the value $\hat\pi_{i,N}$ is proportional to the posterior density $\pi(Z_i|X_1,\cdots,X_n)$
evaluated at point $\theta = Z_i$.
The quantity $h$ is the smoothing bandwidth in the kernel density estimation. 
Because this is a density estimation problem, we would recommend to choose it using
the Silverman's rule of thumb \cite{silverman1986density} or other popular approaches such as 
least square cross-validation \cite{rudemo1982empirical, bowman1984alternative}; see the review paper
of \cite{sheather2004density} for a list of reliable methods.

With the output from Algorithm~\ref{alg::alg::AMP},
the posterior density $\pi(\theta|X_1,\cdots,X_n)$
is represented by 
the collection of points $Z_1,\cdots, Z_N$
along with the corresponding weights $\hat \omega_{1,N},\cdots, \hat \omega_{N,N}$. 
The posterior mean can be approximated using
$$
\hat\theta_{\sf Pmean}= \frac{\sum_{i=1}^n \hat \omega_{i,N} Z_i}{\sum_{i=1}^n \hat \omega_{i,N}}.
$$
This estimator is essentially the importance sampling estimator. 
The posterior mode (MAP: maximum a posteriori) can be approximated using 
$$
\hat \theta_{\sf MAP} = Z_{i^*},\quad i^* =  {\sf argmax}_{i\in \{1,\cdots, N\}} \hat\pi_{i,N}.
$$

The weighted point cloud also leads to an approximated credible region.
Let $1-\alpha$ be the credible level
and $Z_{(1)},\cdots, Z_{(N)}$ be the ordered points such that 
$$
\hat\pi_{(1),N}\geq \hat\pi_{(2),N}\geq \cdots\geq \hat\pi_{(N),N}.
$$
Define 
\begin{equation}
i(\alpha) = {\sf argmin}\left\{i: \frac{\sum_{j=1}^i \hat \omega_{(j),N}}{\sum_{\ell=1}^N \hat \omega_{(\ell),N}}\geq 1-\alpha\right\}.
\label{eq::ialpha}
\end{equation}
Then we may use the collection of points 
\begin{equation}
\{Z_{(1)},\cdots ,Z_{(i(\alpha))}\}
\label{eq::CI1}
\end{equation}
as an approximation of a $1-\alpha$ credible region. 
Alternatively,
one may use the set
$$
\{\theta\in M: \pi(\theta|X_1,\cdots,X_n)\geq \pi(Z_{i(\alpha)}|X_1,\cdots, X_n)\}
$$
as another approximation of a $1-\alpha$ credible region.

Here is an explanation of the choice in equation \eqref{eq::ialpha}. 
$\hat \pi_{i,N}$ is proportional to the posterior value at $Z_i$;
hence, 
$$
\pi(Z_{(1)}|X_1,\cdots,X_n)\geq \pi(Z_{(2)}|X_1,\cdots,X_n)\geq \cdots \geq \pi(Z_{(N)}|X_1,\cdots,X_n).
$$
Define the upper-level set of level $\lambda$ of the posterior distribution as 
$$
L(\lambda) = \{\theta: \pi(\theta|X_1,\cdots,X_n)\geq \lambda\}.
$$
The posterior probability within $L(\lambda)$ is 
$$
\pi(L(\lambda)|X_1,\cdots, X_n) = \int I(\theta\in L(\lambda)) \pi(\theta|X_1,\cdots,X_n) d\theta.
$$
A $1-\alpha$ credible region can be constructed by choosing the minimal value $\lambda_\alpha$ such that 
\begin{equation}
\lambda_\alpha = \inf\left\{\lambda: \pi(L(\lambda)|X_1,\cdots, X_n)  \geq 1-\alpha\right\}.
\label{eq::level}
\end{equation}
With the weights $\hat \omega_{1,N},\cdots,\hat \omega_{N,N}$, 
an approximation to $\pi(L(\lambda)|X_1,\cdots, X_n)$ 
is 
$$
\hat \pi(L(\lambda)|X_1,\cdots, X_n) = \frac{\sum_{j=1}^n \hat \omega_{j,N} I(Z_j \in L(\lambda))}{\sum_{\ell=1}^n \hat \omega_{\ell,N}}.
$$
The posterior levels $\hat \pi_{1,N},\cdots, \hat \pi_{N,N}$ form a discrete approximation of all levels of $\lambda$.
Thus, an approximation to equation \eqref{eq::level} is 
\begin{align*}
\hat \lambda_\alpha &= \min \left\{\hat \pi_{i,N}: \frac{\sum_{j=1}^n \hat \omega_{j,N} I(Z_j \in L(\hat \pi_{i,N}))}{\sum_{\ell=1}^n \hat \omega_{\ell,N}}\geq 1-\alpha\right\}\\
& = \min \left\{\hat \pi_{i,N}:\frac{\sum_{j=1}^n \hat \omega_{j,N} I(\hat \pi_{j,N}\geq\hat \pi_{i,N})}{\sum_{\ell=1}^n \hat \omega_{\ell,N}}\geq 1-\alpha\right\}\\
& = \min \left\{\hat \pi_{(i),N}:\frac{\sum_{j=1}^i \hat \omega_{(j),N} }{\sum_{\ell=1}^n \hat \omega_{(\ell),N}}\geq 1-\alpha\right\}\\
& = \hat \pi_{(i(\alpha)),N},\quad i(\alpha) = {\sf argmin}\left\{i: \frac{\sum_{j=1}^i \hat \omega_{(j),N}}{\sum_{\ell=1}^N \hat \omega_{(\ell),N}}\geq 1-\alpha\right\}.
\end{align*}
The choice in equation \eqref{eq::ialpha} is from the above approximation to the level $\lambda_\alpha$. 


\begin{remark}
Note that the posterior mean may not be on the manifold.
One may replace it by the posterior \emph{Fr{\'e}chet mean} \cite{grove1973conjugatec}
defined as 
$$
\hat\theta_{\sf PFmean} = Z_{i^\dagger},\quad i^\dagger = {\sf argmin}_{i\in \{1,\cdots, N\}} \sum_{j=1}^N \hat\omega_{j,N} (Z_i-Z_j)^2.
$$
The Fr{\'e}chet mean defines a mean of a random variable $X$ using 
the minimization problem ${\sf argmin}_\mu E(X-\mu)^2$
and constraints the minimizer to be in the manifold. 
Here, we use the weighted point approximation to this minimization. 
\end{remark}

\section{Derivation of equation (2)}	\label{sec::fair}

To derive the constraint in equation \eqref{eq::fair2} from equation \eqref{eq::fair},
we expand the first term and consider $s=1$:
\begin{align*}
P(Y=1|Q=1,A=0) & = \frac{P(Y=1,Q=1|A=0)}{P(Q=1|A=0)}\\
& = \frac{\sum_wP(Y=1,Q=1, W=w|A=0)}{\sum_{w'}P(Q=1, W=w'|A=0)}\\
& \overset{(*)}{=} \frac{\sum_wP(Q=1| W=w,A=0) P(W=w,Y=1|A=0)}{\sum_{w'}P(Q=1| W=w',A=0)P(W=w'|A=0)}\\
& = \frac{\sum_w q_{w,0}P(W=w,Y=1|A=0)}{\sum_{w'}q_{w',0}P(W=w'|A=0)}.
\end{align*}
Note that the equality labeled with (*) is due to $Q\perp Y|A,W$.
The two probabilities $P(W=w,Y=1|A=0)$ and $P(W=w'|A=0)$ are identifiable from the data.
A similar calculation shows that 
$$
P(Y=1|Q=1,A=1)  = \frac{\sum_w q_{w,1}P(W=w,Y=1|A=1)}{\sum_{w'}q_{w',1}P(W=w'|A=1)}.
$$
So the test fairness constraint in equation \eqref{eq::fair2} requires
\begin{equation*}
\frac{\sum_w q_{w,0}P(W=w,Y=1|A=0)}{\sum_{w'}q_{w',0}P(W=w'|A=0)} = \frac{\sum_w q_{w,1}P(W=w,Y=1|A=1)}{\sum_{w'}q_{w',1}P(W=w'|A=1)}.
\label{eq::test1}
\end{equation*}
Also, for the case of $s=0$, the above constraint becomes
\begin{equation*}
\frac{\sum_w (1-q_{w,0})P(W=w,Y=1|A=0)}{\sum_{w'}(1-q_{w',0})P(W=w'|A=0)} = \frac{\sum_w (1-q_{w,1})P(W=w,Y=1|A=1)}{\sum_{w'}(1-q_{w',1})P(W=w'|A=1)}.
\label{eq::test2}
\end{equation*}
The above two equations are  what equation \eqref{eq::fair2} refers to as.

\section{Proofs}

%
%
%
%

\begin{proof}[Proof of Lemma~\ref{lem::relax}]
Essentially, we only need to show that when $d(x,M)\leq \frac{3\lambda_M^2}{8\|\Psi\|^*_{\infty,1}\|\Psi\|^*_{\infty,2}}$,
the minimal eigenvalue $\lambda_{\min}(G_{\Psi}(x)G_{\Psi}(x)^T)\geq \frac{1}{4}\lambda_{M}^2$.

For any point $x$ with $d(x,M)\leq \frac{3\lambda_M^2}{8\|\Psi\|^*_{\infty,1}\|\Psi\|^*_{\infty,2}}$, 
let $x_M$ be the projection on $M$. 
The minimal eigenvalue
\begin{equation}
\begin{aligned}
\lambda_{\min}(G_{\Psi}(x)G_{\Psi}(x)^T) &= 
\lambda_{\min}(G_{\Psi}(x_M)G_{\Psi}(x_M)^T) \\
&\quad+ (\lambda_{\min}(G_{\Psi}(x)G_{\Psi}(x)^T) - \lambda_{\min}(G_{\Psi}(x_M)G_{\Psi}(x_M)^T))\\
&\geq \lambda_M^2 - |\lambda_{\min}(G_{\Psi}(x)G_{\Psi}(x)^T) - \lambda_{\min}(G_{\Psi}(x_M)G_{\Psi}(x_M)^T)|.
\end{aligned}
\label{eq::relax::1}
\end{equation}
The Weyl's theorem (see, e.g., Theorem 4.3.1 of \cite{horn2012matrix}) shows that the eigenvalue difference can be bounded via
\begin{align*}
|\lambda_{\min}(G_{\Psi}(x)G_{\Psi}(x)^T) &- \lambda_{\min}(G_{\Psi}(x_M)G_{\Psi}(x_M)^T)|\\
& \leq \|G_{\Psi}(x)G_{\Psi}(x)^T-G_{\Psi}(x_M)G_{\Psi}(x_M)^T\|_{\max}\\
&\leq 2\|\Psi\|^*_{\infty, 1}  \|\Psi\|^*_{\infty, 2}\|x-x_M\|\quad \mbox{(by Taylor's theorem)}\\
&= 2\|\Psi\|^*_{\infty, 1}  \|\Psi\|^*_{\infty, 2}d(x,M).
\end{align*}
Thus, as long as 
$$
2\|\Psi\|^*_{\infty, 1}  \|\Psi\|^*_{\infty, 2}d(x,M)\leq \frac{3}{4} \lambda_M^2
$$
we have $\lambda_{\min}(G_{\Psi}(x)G_{\Psi}(x)^T) \geq \frac{1}{4}\lambda_M^2$,
which completes the proof.

\end{proof}

\begin{proof}[Proof of Theorem~\ref{thm::SM2}]
The proof is modified from the proof of Lemma 4.11 and Theorem 4.12 in \cite{federer1959curvature}. 

Let $r_0 = \min\{\frac{\delta_0}{2}, \frac{\lambda_0}{\norm{\Psi}^*_{\infty, 2}}\}$. 
We will show that $r_0$ is a lower bound of reach$(M)$.
We proceed by the proof of contradiction.

Suppose that the conclusion is incorrect that the reach is less than $r_0$.
Then there exists a point
$x$ such that $d(x,M)< r_0$
and $x$ has two projections onto $M$, denoted as $b,c\in M$. 


Since $b,c\in M$, $\Psi(b)=\Psi(c)=0$ and by Taylor's  remainder theorem and condition (D-2) and (F),
\begin{equation}
\begin{aligned}
\norm{G_\Psi(b)(b-c)}_2& = \norm{f(b)-f(c)- G_\Psi(b)(b-c)}_2\\
&\leq \frac{1}{2}\norm{b-c}^2_2\|\Psi\|^*_{\infty,2}.
\end{aligned}
\label{eq::SM2pf1}
\end{equation}

By the nature of projection, we can find a vector $t_b\in\mathbb{R}^s$ such that $x-b=t_b^TG_{\Psi}(b)$
because the normal space is spanned by the row space of  $G_{\Psi}(b)$ (Lemma~\ref{lem::normal}).
Together with \eqref{eq::SM2pf1}, this implies
\begin{equation}
\begin{aligned}
2|(x-b)^T(b-c)|&=2|t_b^TG_{\Psi}(b)(b-c)|\\
&\leq \norm{G_\Psi(b)(b-c)}_2 \norm{t_b}_2\\
&\leq \|\Psi\|^*_{\infty,2}\norm{b-c}^2_2\norm{t_b}_2.
\label{eq::SM2pf2}
\end{aligned}
\end{equation}

Since $b,c$ are projections of $x$ onto $M$,
\begin{equation*}
\norm{x-b}_2=\norm{x-c}_2.
\end{equation*}
As a result,
\begin{equation}
\begin{aligned}
0& = \norm{x-c}^2_2-\norm{x-b}^2_2\\
& = \norm{b-c}^2_2+2 (b-c)^T(x-b)\\
&\geq \norm{b-c}^2_2-\|\Psi\|^*_{\infty,2}\norm{b-c}^2_2\norm{t_b}_2 \quad\mbox{\eqref{eq::SM2pf2}}\\
& = \norm{b-c}^2_2(1-\|\Psi\|^*_{\infty,2}\norm{t_b}_2).
\label{eq::SM2pf3}
\end{aligned}
\end{equation}

However, 
starting from the definition of $r_0$, we have
\begin{equation}
\begin{aligned}
\frac{\lambda_0}{\|\Psi\|^*_{\infty,2}}> r_0\geq \norm{x-b}_2&= \norm{t^T_bG_\Psi(b)}_2\\
&\underbrace{\geq}_{(F)} \lambda_0 \norm{t_b}_2.
\label{eq::SM2pf4}
\end{aligned}
\end{equation}
As a result,
$$
\|\Psi\|^*_{\infty,2}\norm{t_b}_2 < 1
$$
so 
$\norm{b-c}^2_2=0$ by \eqref{eq::SM2pf3},
which implies $b=c$, a contradiction.
Accordingly, $x$  must have a unique projection so the reach has a lower bound $r_0$.

\end{proof}

\begin{proof}[Proof of Proposittion \ref{prop::sta}]

Consider a point $x \in \tilde{M}$.
By the condition $\norm{\tilde{\Psi}-\Psi}^*_{\infty, 0}<c_0$
and assumption (F), we know that $d(x, M)\leq \delta_0$,
where $c_0$ and $\delta_0$ are the constants in (F). 

Define
\begin{equation}
h(x) =\norm{\Psi(x)}_{2} = \sqrt{\Psi(x)^T \Psi(x)}
\end{equation}
to be the $L_2$ norm for $\Psi$. The derivative of $h(x)$
\begin{equation}
\nabla h(x) = \frac{\Psi(x)^TG_\Psi(x)}{\norm{\Psi(x)}_2} 
\end{equation}
is a vector of $\mathbb{R}^d$. Note that  $G_\Psi(x) = \nabla \Psi(x)\in\R^{s\times d}$ is the Jacobian.

For any point $x\in \tilde{M}$, we define a flow 
\begin{equation}
\phi_x: \mathbb{R} \mapsto \mathbb{R}^d
\end{equation}
such that
\begin{equation}
\begin{aligned}
\phi_x(0) &= x,&\frac{\partial}{\partial t}\phi_x(t) =& -\nabla h(\phi(t)).
\end{aligned}
\end{equation}
Later we will prove in Theorem~\ref{thm::GD} that
$\phi_x(\infty) \in M$ when $x\in M\oplus \delta_c$,
where $\delta_c$ is defined in Theorem~\ref{thm::GD}.

By Theorem 3.39 in \cite{Irwin1980}, $\phi_x(t)$ is uniquely defined since the gradient $\nabla h(x)$ is well-defined for all $x\notin M$. 
We define an arc-length flow (i.e., a constant velocity flow) based on $\phi_x$:
\begin{equation}
\begin{aligned}
\gamma_x(0) &= x,&\frac{\partial}{\partial t}\gamma_x(t) =& -\frac{\nabla h(\gamma_x(t))}{\norm{\nabla h(\gamma_x(t))}_2}.
\end{aligned}
\label{eq::unitG}
\end{equation}
The time traveled in this flow is the same as the distance traveled (due to the velocity being a unit vector).
Let $T_x = \inf\{t>0: \gamma_x(t) \in M\}$ be the terminal time point
and let $\gamma_x(T_x)\in M$ as the endpoint of the flow. 
This means that $T_x$ is the length of the flow from $x$ to the destination on $M$. 
The goal is to bound $T_x$ since the length must be greater or equal to the projection distance
for $x\in \tilde M$. 

We define $\xi_x(t) = h(\gamma_x(t))-h(\gamma_x(T_x)) = h(\gamma_x(t))$.
Differentiating $\xi_x(t)$ with respect to $t$ leads to
\begin{equation}
\begin{aligned}
\xi'_x(t) &= -\frac{d}{dt} h(\gamma_x(t))\\
& = -[\nabla h(\gamma_x(t))]^T \frac{d}{dt}\gamma_x(t)\\
&= - \|\nabla h(\gamma_x(t))\|\\
& = -\frac{\norm{\Psi(\gamma_x(t))^TG_\Psi(\gamma_x(t))}_2}{\norm{\Psi(\gamma_x(t))}_{2}}\\
&\leq -\lambda_{\min}(G_\Psi(\gamma_x(t))G_\Psi(\gamma_x(t))^T)\\
&\leq -\lambda_0
\end{aligned}
\label{eq::solM::pf1}
\end{equation}
because $\gamma_x(t)\in M\oplus \delta_0$ for all $t$.

Let $\epsilon_0 = \norm{\Psi-\tilde{\Psi}}^{*}_{\infty,0}=\sup_x\|\Psi(x)-\tilde\Psi(x)\|_{\max}$ and recall that $x\in \tilde{M}$ so $\tilde \Psi(x) = 0$. Then by the fact that 
$\norm{v}_2\leq \sqrt{d}\times \norm{v}_{\max}$ for vector $v$,
\begin{align*}
\sqrt{d}\cdot \epsilon_0&  =\sqrt{d}\sup_x\|\Psi(x)-\tilde\Psi(x)\|_{\max}\\
&  \geq\sup_x\|\Psi(x)-\tilde\Psi(x)\|\\
&\geq \|\Psi(x)-\tilde\Psi(x)\|\\
&\geq h(x)\\
& =h(\gamma_x(0))-h(\gamma_x(T_x))\quad \mbox{(since $h(\gamma_x(T_x))=0$)}\\
& =\xi(0)-\xi(T_x)\quad \mbox{($\xi(T_x)=0$ and $\xi(0)=h(0)$)}\\
& = -T_x\xi'(T_x^*)\quad \mbox{(mean value Theorem)}\\
&\geq T_x\lambda_0\quad \mbox{by equation \eqref{eq::solM::pf1}}.
\end{align*}
Hence, $T_x\leq\frac{\sqrt{d}}{\lambda_0}\epsilon_0 = O(\epsilon_0)$ which is independent of $x$. This implies that 
$$
\sup_{x\in\tilde{M}} d(x,M)\leq \frac{\sqrt{d}}{\lambda_0}\epsilon_0 = O(\|\tilde \Psi-\Psi\|^*_{\infty,0}).
$$

\end{proof}

\begin{proof}[Proof of Theorem \ref{thm::LU0}]
\begin{itemize}
\item[1.] 
Since condition (F) involves only $\Psi$ and its derivative, when $\norm{\Psi-\tilde{\Psi}}^*_{\infty,2}$ is sufficiently small, (F) holds for $\tilde{\Psi}$.

\item[2.] 

By the first assertion, condition (F) holds for $\tilde\Psi$. 

Thus, we can exchange $\tilde{M}$ and $M$ and repeat the proof of Proposittion \ref{prop::sta}, which leads to
$$
\sup_{x\in M} d(x,\tilde{M})\leq \frac{\sqrt{d}}{\lambda_0}\epsilon_0.
$$
As a result, we conclude that $\Haus(\tilde{M},M)\leq \frac{\sqrt{d}}{\lambda_0}\epsilon_0 = O(\epsilon_0)$.

\item[3.] By Theorem~\ref{thm::SM2}, the reach of $M$ has lower bound $\min\{\delta_0/2, \lambda_0/\|\Psi\|^*_{\infty,2}\}$. Note that $\delta_0,\lambda_0$ depends on the first derivative of $\Psi$. Hence, the lower bound for reach of $M$ and $\tilde{M}$ will be bounded at rate $O(\norm{\Psi-\tilde{\Psi}}^*_{\infty,2})$.

\end{itemize}


\end{proof}

Before moving forward, we would like to note that the Jacobian and Hessian of $f$ can be expressed as 
\begin{align}
G_f(x) = \nabla f(x) &= 2\Psi(x)^T \Lambda [\nabla \Psi(x)]
\label{eq::GD::F}\\
H_f(x) = \nabla\nabla f(x) &= 2[\nabla \Psi(x)]^T \Lambda [\nabla \Psi(x)] + 2\Psi(x)\Lambda [\nabla\nabla \Psi(x)],
\label{eq::H::F}\\
\nabla\nabla\nabla f(x) &= 6[\nabla \Psi(x)]^T \Lambda [\nabla\nabla \Psi(x)] + 2\Psi(x)\Lambda [\nabla\nabla\nabla \Psi(x)],
\label{eq::T::F}
\end{align}
where $\nabla \Psi(x) \in \R^{s\times d}$ and $\nabla \nabla \Psi(x) \in \R^{s\times d\times d}$.

\begin{proof}[Proof of Lemma~\ref{lem::property2}]

{\bf Property 1 (For each $x\in M$).}

\emph{1-(a).}
By equation \eqref{eq::H::F} and the fact that $\Psi(x) = 0$ whenever $x\in M$, 
we obtain 
$$
H_f(x) = 2[\nabla \Psi(x)]^T \Lambda [\nabla \Psi(x)].
$$
Because $\Lambda$ is positive definite, it can be decomposed into $\Lambda = U D U^T$ where $D$ is a diagonal matrix 
so the eigenvectors corresponding to non-zero eigenvalues of $H_f$
will be the the rows of $U^T[\nabla \Psi(x)]$, which spans the same subspace as the row space of $[\nabla \Psi(x)]$
so by Lemma~\ref{lem::normal}, the non-zero eigenvectors spans the normal space of $M$ at $x$.

\emph{1-(b).}
Because $H_f(x) = 2[\nabla \Psi(x)]^T \Lambda [\nabla \Psi(x)]$ when $x\in M$, 
the minimal non-zero eigenvalue
\begin{align*}
\lambda_{\min,>0} (H_f(x)) = 2\lambda_{\min,>0} (G_{\Psi}(x)^T \Lambda G_{\Psi}(x)).
\end{align*}
Since $\Lambda$ is positive definite and symmetric, we can decompose 
$$
G_{\Psi}(x)^T \Lambda G_{\Psi}(x) = G_{\Psi}(x)^T \Lambda^{1/2} \Lambda^{1/2} G_{\Psi}(x)
$$
so we obtain 
\begin{align*}
\lambda_{\min,>0} (H_f(x)) &= 2\lambda_{\min,>0} (G_{\Psi}(x)^T \Lambda G_{\Psi}(x)) \\
&= 2\lambda_{\min} ( \Lambda^{1/2}G_{\Psi}(x)G_{\Psi}(x)^T  \Lambda^{1/2} )\\
&\geq 2 \Lambda_{\min} \lambda_{\min}(G_{\Psi}(x)G_{\Psi}(x)^T)\\
&\geq 2\Lambda_{\min} \lambda_0^2.
\end{align*}

\emph{1-(c).}
Because the normal space of $M$ at $x$ is spanned by the rows of $G_\Psi(x) = \nabla \Psi(x)$, 
which by 1-(a) is spanned by the non-zero eigenvectors of $H_f(x)$, 
the result follows. 

{\bf Property 2.}
Because $d(x,M)<\delta_c$ so $x$ is within the reach of $M$ and thus, $x_M$, the projection from $x$ onto $M$, is unique. 
As a result, the normal space of $x_M$, $N_M(x)$, is well-defined.
 
We can decompose 
\begin{align*}
\lambda_{\min, \perp, M}(H_f(x)) & = \min_{v\in N_M(x)}\frac{v^T H_f(x)v}{\|v\|^2}\\
& = \min_{v\in N_M(x)}\frac{v^T (H_f(x_M) + H_f(x) - H_f(x_M) )v}{\|v\|^2}\\
&\geq \min_{v\in N_M(x)}\frac{v^T H_f(x_M)v}{\|v\|^2} - \max_{v\in N_M(x)}\frac{v^T (H_f(x) - H_f(x_M))v}{\|v\|^2}\\
&\geq \min_{v\in N_M(x)}\frac{v^T H_f(x_M)v}{\|v\|^2} -d\|H_f(x) - H_f(x_M)\|_{\max}\\
&\geq 2 \lambda_0^2 \Lambda_{\min}-d\|H_f(x) - H_f(x_M)\|_{\max}\\
\end{align*}
By equation \eqref{eq::H::F}, 
\begin{align*}
\|H_f(x) - H_f(x_M)\|_{\max} 
&\leq 2\|G_{\Psi}(x)^T \Lambda G_{\Psi}(x)-G_{\Psi}(x_M)^T \Lambda G_{\Psi}(x_M)\|_{\max} \\
&+ 2\|\Psi(x)\Lambda H_{\Psi}(x) -\Psi(x_M)\Lambda H_{\Psi}(x_M)\|_{\max}\\
&\leq 4\|\Psi\|^*_{\infty, 1} \Lambda_{\max} \|\Psi\|^*_{\infty, 2}\|x-x_M\| 
+ 4\|\Psi\|^*_{\infty,2} \Lambda_{\max} \|\Psi\|^*_{\infty, 3}\|x-x_M\| \\
&\leq 8 \|\Psi\|^*_{\infty,2} \Lambda_{\max} \|\Psi\|^*_{\infty, 3}\|x-x_M\|.
\end{align*}
Thus, as long as 
$$
\|x-x_M\| = d(x,M) \leq \frac{\lambda_0^2 \Lambda_{\min}}{8d \Lambda_{\max}\|\Psi\|^*_{\infty,2}  \|\Psi\|^*_{\infty, 3}},
$$
we have 
$$
\lambda_{\min, \perp, M}(H_f(x)) \geq  \lambda_0^2 \Lambda_{\min},
$$
which completes the proof.

\end{proof}

\begin{proof}[Proof of Theorem~\ref{thm::GD}]


{\bf 1. Convergence radius.}
We prove this by showing that for any $x\in M\oplus \delta_c$  and $x\notin M$, the destination $\pi_x(\infty)\in M$. 
The idea of the proof relies on two properties:
\begin{itemize}
\item[(P1)] Any stationary point of $f$ inside $M\oplus \delta_c$ must be a point in $M$. 

\item[(P2)]
Let $x_M$ be a point on $M$ that is closest to $x$.
For any point $x\in M\oplus \delta_c$, $(x-x_M)^T \nabla f(x)>0$. 
Namely, the gradient flow only moves $\pi_x(t)$ closer toward $M$. 
\end{itemize}
With the above two properties, it is easy to see that if we start a gradient flow $\pi_x$ from $x\in M\oplus \delta_c$, 
then by (P2) this flow must stays within $M\oplus \delta_c$. 
Because stationary points within $M\oplus \delta_c$ are all in $M$ by (P1)
and the destination of a gradient flow must be a stationary point,
we conclude that $\pi_x(\infty)\in M$, which completes the proof of convergence radius. 
In what follows, we show the two properties.

{\bf Property P1: Any stationary point inside $M\oplus \delta_c$ must be a point in $M$.}
Because $\nabla f(x) = \Psi(x)\Lambda G_{\Psi}(x)$ and $\Lambda$ is positive definite,
there are only two cases that $\nabla f(x) = 0$: 1. $\Psi(x)= 0$ and 2., row space of $G_{\Psi}(x)$
has a dimension less than $s$ (in fact, if $\Psi(x)\neq 0$, then the second case is a necessary condition). 
The first case is the solution manifold $M$
so we only need to focus on showing that 
the second case will not happen for $x\in M\oplus \delta_c$.

The row space of $G_{\Psi}(x)$ has a dimension less than $s$
when there exists a singular value of $G_{\Psi}(x)$ being $0$;
or equivalently, $\lambda_{\min}(G_{\Psi}(x)G_{\Psi}(x)^T) = 0$.
However, assumption (F) already requires that this will not happen within $M\oplus  \delta_c$.
Thus, this property holds.

%

%

{\bf Property P2: For any $x\in M\oplus \delta_c$, the directional gradient $(x-x_M)^T \nabla f(x)>0$.}
By Taylor expansion and property 2 of Lemma~\ref{lem::property2}, 
\begin{equation}
\begin{aligned}
(x-x_M)^T \nabla f(x) &= (x-x_M)^T (\nabla f(x) - \underbrace{\nabla f(x_M)}_{=0})\\
& = (x-x_M)^T\int_{\epsilon=0}^{\epsilon=1} H_f(x_M +\epsilon(x-x_M)) (x-x_M) d\epsilon\\
&\geq \|x-x_M\|^2 \inf_{y\in M\oplus \delta_c}\lambda_{\min,\perp, M}(H_f(y))\\
&\geq d(x,M)^2 \lambda_0^2 \Lambda_{\min}>0
\end{aligned}
\label{eq::GD::reverse}
\end{equation}


{\bf 2. Terminal flow orientation.}
To study the gradient flow close to $M$, it suffices to analyze the behavior of gradient close to $M$.
Let $x\in M$ and define $u $ to be a unit vector in the normal space of $M$ at $x$. 
By Lemma~\ref{lem::normal}, 
$u$ belongs to the row space of $\nabla \Psi(x)=G_{\Psi}(x)$. 

Now we consider the gradient at $x+\epsilon u$ when $\epsilon\rightarrow0$. 
By Taylor's theorem and the fact that $f$ has bounded third derivatives (from (D-3)), 
$$
G_f(x+\epsilon u) \equiv \nabla f(x+\epsilon u) = \nabla f(x+\epsilon u)-\nabla f(x) = \epsilon H_f(x) u + O(\epsilon^2).
$$
Thus, 
$$
\lim_{\epsilon\rightarrow0} \frac{1}{\epsilon}G_f(x+\epsilon u) = H_f(x) u.
$$
By equation \eqref{eq::H::F},
$$
H_f(x) = 2G_\Psi(x)^T \Lambda G_\Psi(x) + 2\Psi(x)\Lambda H_\Psi(x) = 2G_\Psi(x)^T \Lambda G_\Psi(x)
$$
because $\Psi(x) = 0$ when $x\in M$.
Using the fact that $G_\Psi(x)^T = [\nabla \Psi_1(x),\cdots, \nabla\Psi_s(x)]$, 
it is easy to see that
$$
H_f(x) u = \sum_{\ell=1}^s a_\ell \nabla \Psi_\ell(x),
$$
where $a_\ell  = e_\ell^T \Lambda G_\Psi(x)u$ with $e_\ell = (0,0,\cdots,0,1,0,\cdots, 0)^T \in \R^s$ is the coordinate vector pointing toward $\ell$-th coordinate. 
Thus, by Lemma~\ref{lem::normal} $\nabla \nabla f(x) u$ belongs to the normal space of $M$ at $x$,
which completes the proof of terminal orientation.

\end{proof}

\begin{proof}[Proof of Theorem~\ref{thm::AMT}]

We prove this result using the idea of the Lyapunov-Perron method \cite{perko2013differential}.
Recall that $A(z) = \{x: \pi_x(\infty) = z\}$ for $z\in M$
is the basin of attraction of point $z$. 
Consider a ball $B(z, r)$ such that any gradient flow $\pi_x(t)$ that converges to $z = \pi_x(\infty)$ 
intersects one and only one point at the boundary $\partial B(z,r) = \{y: \|y-z\| = r\}$.
This occurs when $r<\delta_c$ 
due to property (P2) in the proof of Theorem~\ref{thm::GD}.

Consider the gradient flow $\pi_x(t)$ with $x\in \partial B(z,r)$ and $\pi_x(\infty) = z$.
By Taylor's theorem,
this flow solves the following equation
\begin{equation}
\begin{aligned}
\pi'_x(t) = -G_f(\pi_x(t))  &= -G_f(\pi_x(t))  + \underbrace{G_f(\pi_x(\infty)) }_{=0}\\
& = -H_f(\pi_x(\infty)) (\pi_x(t) - \pi_x(\infty)) + \epsilon(\pi_x(t)),
\end{aligned}
\label{eq::AMTpf1}
\end{equation}
where $\|\epsilon(\pi_x(t))\|\leq C_0 \|\pi_x(t) - \pi_x(\infty)\| \leq C_0 r$ for some finite constant $C_0$
due to Assumption (D-3).
Equation \eqref{eq::AMTpf1} is a perturbed ODE with a fixed point $\pi_x(\infty)$ and by the variation of parameters,
its solution can be written as 
$$
\pi_x(t) - \pi_x(\infty) = e^{- t H_f(\pi_x(\infty))} (\pi_x(0) - \pi_x(\infty)) + \int_{s=0}^{s=t} e^{- (t-s) H_f(\pi_x(\infty))} \epsilon(\pi_x(s))ds.
$$
Denoting $v_x = \pi_x(0) - \pi_x(\infty) $,
we can rewrite the flow as 
$$
\pi_x(t) -\pi_x(\infty) =e^{- t H_f(\pi_x(\infty))} v_x + \int_{s=0}^{s=t} e^{- (t-s)H_f(\pi_x(\infty))} \epsilon(\pi_x(s))ds.
$$
By Lemma~\ref{lem::normal}, the normal space of $M$ at $z = \pi_x(\infty)$ is the row space of $G_{\Psi}(z) = \nabla\Psi(z)$,
which will also be the space spanned by the eigenvectors of $H_f(\pi_x(\infty))$ 
that corresponds to non-zero eigenvalues (Lemma~\ref{lem::property2} 1-(a)). 
The spectral decomposition shows
$H_f(\pi_x(\infty)) = \sum_{\ell=1}^s\lambda_\ell u_\ell u_\ell^T$
and we define 
the projection matrix onto the normal space of  $M$ as $\Pi_N = \sum_{\ell=1}^s u_\ell u_\ell^T$
and the projection matrix onto the tangent space of $M$
as $\Pi_T = I_d - \Pi_N$.
By construction, $\Pi_N H_f(\pi_x(\infty)) = H_f(\pi_x(\infty))$
and $\Pi_T H_f(\pi_x(\infty)) = 0$
so $\Pi_Ne^{-t H_f(\pi_x(\infty))} = e^{-t H_f(\pi_x(\infty))}$
and $\Pi_T e^{-t H_f(\pi_x(\infty))} = \Pi_T$.

We decompose 
\begin{equation}
\begin{aligned}
\pi_x(t) - \pi_x(\infty) &= \Pi_T(\pi_x(t) - \pi_x(\infty)) + \Pi_N (\pi_x(t) - \pi_x(\infty))\\
&= \Pi_Te^{- t H_f(\pi_x(\infty))} v_x + \Pi_T \int_{s=0}^{s=t} e^{- (t-s)H_f(\pi_x(\infty))} \epsilon(\pi_x(s))ds\\
&\quad + \Pi_Ne^{- t H_f(\pi_x(\infty))} v_x + \Pi_N \int_{s=0}^{s=t} e^{- (t-s)H_f(\pi_x(\infty))} \epsilon(\pi_x(s))ds\\
& = v_{x,T} + \int_{s=0}^{s=t}  \epsilon_T(\pi_x(s))ds
+e^{- t H_f(\pi_x(\infty))} v_{x,N} \\
&\quad + \int_{s=0}^{s=t} e^{- (t-s)H_f(\pi_x(\infty))} \epsilon_N(\pi_x(s))ds,
\end{aligned}
\label{eq::AMTpf2}
\end{equation}
where 
\begin{align*}
v_{x,T} = \Pi_T v_x, \quad v_{x,N} = \Pi_N v_x,
\quad \epsilon_T(\pi_x(s)) = \Pi_T  \epsilon(\pi_x(s)),
\quad\epsilon_N(\pi_x(s)) = \Pi_N  \epsilon(\pi_x(s)).
\end{align*}

In the tangent direction, when $t\rightarrow \infty$
\begin{align*}
0 &=\lim_{t\rightarrow\infty} \Pi_T(\pi_x(t) - \pi_x(\infty)) \\
&= \lim_{t\rightarrow\infty} \Pi_T e^{- t H_f(\pi_x(\infty))}v_x +\lim_{t\rightarrow\infty}  \Pi_T\int_{s=0}^{s=t} e^{- (t-s) H_f(\pi_x(\infty))} \epsilon(\pi_x(s))ds\\
& = \Pi_T v_x + \int_{s=0}^{s=\infty} \Pi_T\epsilon(\pi_x(s))ds\\
& = v_{x,T} + \int_{s=0}^{s=\infty} \epsilon_T(\pi_x(s))ds.
\end{align*}
Thus, 
\begin{equation}
v_{x,T} = - \int_{s=0}^{s=\infty} \epsilon_T(\pi_x(s))ds
\label{eq::AMTpf3}
\end{equation}
and equation \eqref{eq::AMTpf2} can be rewritten as 
\begin{equation}
\begin{aligned}
\pi_x(t) - \pi_x(\infty) &= - \int_{s=0}^{s=\infty} \epsilon_T(\pi_x(s))ds + \int_{s=0}^{s=t}  \epsilon_T(\pi_x(s))ds\\
&\quad+e^{- t H_f(\pi_x(\infty))} v_{x,N} 
+ \int_{s=0}^{s=t} e^{- (t-s)H_f(\pi_x(\infty))} \epsilon_N(\pi_x(s))ds\\
& = e^{- t H_f(\pi_x(\infty))} v_{x,N} 
+ \int_{s=0}^{s=t} e^{- (t-s)H_f(\pi_x(\infty))} \epsilon_N(\pi_x(s))ds\\
&\qquad -\int_{s=t}^{s=\infty}  \epsilon_T(\pi_x(s))ds.
\end{aligned}
\label{eq::AMTpf4}
\end{equation}
The latter two terms involving integral are determined entirely by the 
Taylor remainder terms $\epsilon(\pi_x(t))$. 
Thus, to uniquely determine a point on the gradient flow $\pi_x(t)$
that converges to $z$ (and is inside $B(z,r)$),
we only need to specify 
the time $t$ and the vector $v_{x,N} $ that belongs to the normal space of $M$ at $z$
with $\|v_{x,N} \| = r$. 
Namely,  there exists a mapping (due to equation \eqref{eq::AMTpf4}) $\Omega$ such that
$$
\pi_x(t) = \Omega(t, v_{x,N})
$$
for all $\pi_x(t)$ with $\|x-z\| = r$. 
Note that equation \eqref{eq::AMTpf4} implies that the mapping $\Omega$ has bounded derivative with respect to both $t$ and $v_{x,N}$.
Therefore, 
the set
$$
A(z)\cap B(z,r) = \left\{\pi_x(t) = \Omega(t,v_{x,N}): t\in [0,\infty), v_{x,N} = \sum_{\ell=1}^s a_\ell u_\ell, \sum_{\ell=1}^s a_\ell^2 = r^2 \right\}
$$
is parameterized by $(t,a_1,\cdots, a_s)$ with a constraint $\sum_{\ell=1}^s a_\ell^2 = r^2$
so it is an $s$-dimensional manifold. 

To generalize this to the entire set $A(z)$,
note that every gradient flow ending at $z$ must pass the boundary $\partial B(z,r)$
so allowing the gradient $\pi_x(t)$ to move toward $t\rightarrow-\infty$
covers the entire basin, i.e.,
$$
A(z) = \left\{\pi_x(t) = \Omega(t,v_{x,N}): t\in \R, v_{x,N} = \sum_{\ell=1}^s a_\ell u_\ell, \sum_{\ell=1}^s a_\ell^2 = r^2 \right\}.
$$
This implies that $A(z)$ is parametrized by $(t,a_1,\cdots, a_s)$ with a constraint $\sum_{\ell=1}^s a_\ell^2 = r^2$
so again it is an $s$-dimensional manifold.



\end{proof}

\begin{proof}[Proof of Theorem~\ref{thm::alg}]



{\bf Convergence of $f(x_t)$.}
Because $x_{t+1} = x_t - \gamma G_f(x_t)$, simple Taylor expansion shows that
\begin{align*}
f(x_t)- f(x_{t+1}) &= f(x_t) - f(x_t -\gamma G_f(x_t))\\
&= \gamma \|G_f(x_t)\|^2 - \frac{1}{2}\gamma^2 \int_{\epsilon=0}^{\epsilon=1} G_f(x_t) H_f(x_t-\epsilon\gamma G_f(x_t)) G_f(x_t) d\epsilon\\
&\geq \gamma \|G_f(x_t)\|^2 -\frac{1}{2} \gamma^2 \|G_f(x_t)\|^2\sup_z\|H_f(z)\|_{2}.
\end{align*}
Note that one can also use the fact that the gradient $G_f$ is Lipschitz to obtain a similar bound.
%

Thus, when $\gamma < \frac{2}{\sup_z\|H_f(z)\|_{2}}$, 
we obtain
$$
f(x_t)- f(x_{t+1}) >0
$$
which implies that $f(x_{t+1})< f(x_t)$, i.e.,
the objective function is decreasing. 
We can summarize the result as 
\begin{equation}
f(x_{t+1}) \leq f(x_t) - \gamma \|G_f(x_t)\|^2\left(1 -\frac{1}{2} \gamma \sup_z\|H_f(z)\|_{2}\right).
\label{eq::alg::pf1}
\end{equation}

To obtain the algorithmic convergence rate, we need to associate the objective function $f(x)$ and the squared gradient $\|G_f(x)\|^2$. 
We focus on the case of $t=0$, and investigate 
\begin{equation}
f(x_{1}) \leq f(x_0) - \gamma \|G_f(x_0)\|^2\left(1 -\frac{1}{2} \gamma \sup_z\|H_f(z)\|_{2}\right).
\label{eq::alg::pf2}
\end{equation}

Because $d(x_0,M)\leq \delta_c\leq {\sf reach}(M)$, there is a unique projection $x_M\in M$ from $x_0$.
Note that $d(x_0,M) = \|x_0-x_M\|$.
The gradient has a lower bound from the following Taylor expansion:
\begin{align*}
\|G_f(x_0)\| &= \|G_f(x_0)- \underbrace{G_f(x_M)}_{=0}\|\\
& =\left\| \int_{\epsilon=0}^{\epsilon=1} H_f(x_M + \epsilon(x_0-x_M)) (x_0-x_M)d\epsilon\right\|\\
&\geq \|x_0-x_M\| \inf_{\epsilon\in [0,1]}\lambda_{\min,\perp}(H_f(x_M + \epsilon(x_0-x_M)))\\
&\geq \|x_0-x_M\| \lambda^2_0 \Lambda_{\min}\\
&\geq d(x_0,M) \lambda^2_0 \Lambda_{\min},
\end{align*}
where the second to the last inequality is due to property 2 in  Lemma~\ref{lem::property2}.
Thus, 
\begin{equation}
\|G_f(x_0)\|^2 \geq d(x_0,M)^2 \lambda^4_0 \Lambda^2_{\min}.
\label{eq::alg::pf3}
\end{equation}

The distance $d(x_0,M)$ and the objective function $f(x_0)$
can also be associated using another Taylor expansion:
\begin{align*}
f(x_0) & = f(x_0) - f(x_M)\\
& = (x_0-x_M)^T \underbrace{G_f(x_M)}_{=0} +\frac{1}{2} (x_0-x_M)^T \int_{\epsilon=0}^{\epsilon=1} H_f(x_M+\epsilon (x_0-x_M)) d\epsilon (x_0-x_M)\\
&\leq \frac{1}{2}d^2(x_0,M) \sup_z\|H_f(z)\|_{2}.
\end{align*}
Thus, 
\begin{equation*}
d^2(x_0,M) \geq \frac{2f(x_0)}{\sup_z\|H_f(z)\|_{2}}
\end{equation*}
which implies an improved bound on equation \eqref{eq::alg::pf3} as 
\begin{equation}
\|G_f(x_0)\|^2 \geq d(x_0,M)^2\lambda^4_0 \Lambda^2_{\min} \geq \frac{ \lambda^4_0 \Lambda^2_{\min}}{\sup_z\|H_f(z)\|_{2}} 2f(x_0).
\label{eq::alg::pf4}
\end{equation}

By inserting equation \eqref{eq::alg::pf4} into equation \eqref{eq::alg::pf2},
we obtain
\begin{align*}
f(x_{1}) &\leq f(x_0) - \gamma \|G_f(x_0)\|^2\left(1 -\frac{1}{2} \gamma \sup_z\|H_f(z)\|_{2}\right)\\
&\leq f(x_0) - \gamma\left(1 -\frac{1}{2} \gamma \sup_z\|H_f(z)\|_{2}\right) \frac{ \lambda^4_0 \Lambda^2_{\min}}{\sup_z\|H_f(z)\|_{2}} 2f(x_0)\\
& = f(x_0)\left(1-2\gamma\left(1 -\frac{1}{2} \gamma \sup_z\|H_f(z)\|_{2}\right) \frac{\lambda^4_0 \Lambda^2_{\min}}{\sup_z\|H_f(z)\|_{2}} \right)
\end{align*}
When $\gamma< \frac{1}{\sup_z\|H_f(z)\|_{2}}$, the above inequality can be simplified as
$$
f(x_{1}) \leq f(x_0) \cdot {\left(1-\gamma \frac{\lambda^4_0 \Lambda^2_{\min}}{\sup_z\|H_f(z)\|_{2}} \right)}.
$$
Thus, we have proved the result for $t=0$. 
The same derivation works for other $t$ (by treating $x_t$ as $x_0$).
By telescoping, we conclude that
$$
f(x_t) \leq f(x_0) \cdot \left(1-\gamma \frac{\lambda^4_0 \Lambda^2_{\min}}{\sup_z\|H_f(z)\|_{2}} \right)^t.
$$
Finally, using the fact  that $\sup_z\|H_f(z)\|_{2}\leq \Lambda_{\max} \|\Psi\|^*_{\infty,2}$, we obtain the desired bound.



{\bf Convergence of $d(x_t,M)$.}
Let $x_{t,M} \in M$ be the point on the manifold that is closest to $x_t$; again, due to the reach condition
this projection is unique.
The Taylor expansion along with property 2 in Lemma~\ref{lem::property2} shows that 
\begin{equation*}
\begin{aligned}
-f(x_t) &= f(x_{t,M}) - f(x_t) \\
&= (x_{t,M}-x_{t})^T G_f(x_t) + \frac{1}{2}(x_{t,M}-x_{t})^T \int_{\epsilon=0}^{\epsilon=1} H_f(x_t+\epsilon(x_{t,M}-x_t)) (x_t-x_{t,M})d\epsilon\\
&\geq (x_{t,M}-x_{t})^T G_f(x_t)  + \frac{1}{2}\|x_t - x_{t,M}\|^2 \lambda_0^2 \Lambda_{\min}. 
\end{aligned}
\end{equation*}
Thus, 
\begin{equation}
-f(x_t) -  \frac{1}{2}\|x_{t,M} - x_{t}\|^2 \lambda_0^2 \Lambda_{\min}  \geq - (x_{t}-x_{t,M})^T G_f(x_t). 
\label{eq::LC::01}
\end{equation}
Because of equation \eqref{eq::alg::pf1} and that $\sup_z\|H_f(z)\|_{2}\leq d\|\Psi\|^*_{\infty,2}$, 
we have 
\begin{align*}
f\left(x - \frac{1}{d\Lambda_{\max} \|\Psi\|^*_{\infty,2}} G_f(x)\right)-f(x)&\leq -\frac{1}{2d\Lambda_{\max} \|\Psi\|^*_{\infty,2}}\|G_f(x)\|^2.
\end{align*}
Using the fact that $f\left(x - \frac{1}{d\Lambda_{\max} \|\Psi\|^*_{\infty,2}} G_f(x)\right)\geq 0$, 
we conclude that
\begin{equation*}
\frac{1}{2d\Lambda_{\max} \|\Psi\|^*_{\infty,2}}\|G_f(x)\|^2 \leq  f(x),
\end{equation*}
which implies
\begin{equation}
\|G_f(x)\|^2 \leq  2d\Lambda_{\max} \|\Psi\|^*_{\infty,2}f(x).
\label{eq::LC::02}
\end{equation}


For any $t$, we have
\begin{align*}
d(x_{t+1},M)^2 &\leq 
\|x_{t+1} - x_{t,M}\|^2 \\
&= \|x_{t} - x_{t,M}-\gamma G_f(x_t)\|^2\\
& = \|x_{t}- x_{t,M}\|^2 - 2(x_t-x_{t,M})^TG_f(x_t) + \gamma^2\|G_f(x_t)\|^2\\
&\overset{\eqref{eq::LC::01}}{\leq}\|x_{t}- x_{t,M}\|^2 (1-\gamma \lambda_0^2 \Lambda_{\min}) -2\gamma f(x_t) + \gamma^2 \|G_f(x_t)\|^2\\
& \overset{\eqref{eq::LC::02}}{\leq}\|x_{t}- x_{t,M}\|^2 (1-\gamma \lambda_0^2 \Lambda_{\min}) - \underbrace{2\gamma f(x_t) \left(1-d\gamma \Lambda_{\max} \|\Psi\|^*_{\infty,2}\right)}_{\geq 0}\\
&\leq \|x_{t}- x_{t,M}\|^2 (1-\gamma \lambda_0^2 \Lambda_{\min})\\
&= d(x_t,M)^2(1-\gamma \lambda_0^2 \Lambda_{\min})
\end{align*}
whenever $\gamma<\frac{1}{\Lambda_{\max} \|\Psi\|^*_{\infty,2}}$.
By telescoping, the result follows.


\end{proof}



\bibliographystyle{imsart-number.bst}
\bibliography{SolM_GD.bib}


\end{document}